\documentclass[a4paper,12pt,reqno]{amsart}
\usepackage[utf8]{inputenc}
\usepackage[pdftex]{graphicx}
\usepackage{a4wide}
\usepackage[english]{babel}
\usepackage[T1]{fontenc}
\usepackage{amsmath,amssymb,amsthm,stmaryrd}
\usepackage{mathrsfs}
\usepackage{esint}
\usepackage{mathtools}
\usepackage{comment}
\usepackage{color}
\usepackage{extarrows}
\usepackage{tikz}
\usepackage{tcolorbox}
\usepackage{setspace}
\usetikzlibrary{arrows}

\usepackage{pgf}
\usepackage{color}

\DeclareMathOperator*{\fiint}{\ensuremath{\iint\text{\kern-1.36em{\raisebox{5.87pt}{\rotatebox{-93}{$\setminus$}}}}}}

\newcommand{\R}{\mathbb{R}}
\newcommand{\ree}{\mathbb{R}^{n+1}}

\newcommand{\D}{\mathbb{D}}

\newcommand{\N}{\mathbb{N}}
\newcommand{\Z}{\mathbb{Z}}
\newcommand{\Fs}{\mathscr{F}}
\newcommand{\Bc}{\mathcal{B}}
\newcommand{\Fc}{\mathcal{F}}
\newcommand{\Hc}{\mathcal{H}}

\newcommand{\Nc}{\mathcal{N}}
\newcommand{\Sc}{\mathcal{S}}

\newcommand{\Ac}{\mathcal{A}}

\newcommand{\Gc}{\mathcal{G}}
\newcommand{\Ms}{\mathscr{M}}

\newcommand{\Pc}{\mathcal{P}}
\newcommand{\Tc}{\mathcal{T}}

\newcommand{\Wc}{\mathcal{W}}
\newcommand{\Uc}{\mathcal{U}}

\newcommand{\Cs}{\mathscr{C}}

\newcommand{\G}{\mathcal{G}}
\newcommand{\B}{\mathcal{B}}
\newcommand{\om}{\Omega}
\newcommand{\dd}{\mathbb{D}}

\newcommand{\tmf}{{\mathfrak T}}
\newcommand{\W}{\mathcal{W}}

\newcommand{\BMO}{\text{BMO}}

\newcommand{\diam}{\text{diam}}
\newcommand{\dist}{\text{dist}}
\newcommand{\fat}{\text{fat}}
\newcommand{\NT}{\text{NT}}

\newcommand{\eps}{\varepsilon}
\newcommand{\pom}{\partial\Omega}
\newcommand{\oPsi}{\overrightarrow{\Psi}}
\newcommand{\wip}{\widetilde{\partial}}
\newcommand{\Cmf}{\mathfrak{C}}

\newtheorem{theorem}[equation]{Theorem}
\newtheorem*{theorem*}{Theorem}
\newtheorem{lemma}[equation]{Lemma}

\newtheorem{proposition}[equation]{Proposition}

\theoremstyle{definition}
\newtheorem{defin}[equation]{Definition}
\newtheorem*{defin*}{Definition}

\newtheorem*{question*}{Question}

\newtheorem{remark}[equation]{Remark}
\newtheorem*{remark*}{Remark}
\newtheorem{notation}[equation]{Notation}
\numberwithin{equation}{section}

\title{Uniform rectifiability implies Varopoulos extensions}

\author{Steve Hofmann \and Olli Tapiola}

\address{Steve Hofmann, Department of Mathematics, University of Missouri, Columbia, MO 65211, USA}
\email{hofmanns@missouri.edu}

\address{Olli Tapiola, Department of Mathematics and Statistics, P.O. Box 35 (MaD), FI-40014 University of Jyväskyl\"a, Finland}
\email{olli.m.tapiola@gmail.com}

\date{\today}
\subjclass[2010]{}
\thanks{S.H. was supported by NSF grant DMS-1664047. O.T. was partially supported by Emil Aaltosen Säätiö through Foundations' Post Doc Pool grant.}

\begin{document}

\begin{abstract}
  We construct extensions of Varopolous type for functions $f \in \BMO(E)$, for any uniformly rectifiable set $E$ of codimension one.  More precisely, let $\Omega \subset \R^{n+1}$ be an open set satisfying the corkscrew condition, with an $n$-dimensional uniformly rectifiable boundary $\pom$, and let $\sigma \coloneqq \Hc^n\lfloor_{\pom}$ denote the surface measure on $\pom$. We show that if $f \in \BMO(\pom,d\sigma)$ with compact support on $\pom$, then there exists a smooth function $V$ in $\Omega$ such that $|\nabla V(Y)| \, dY$ is a Carleson measure with Carleson norm controlled by the BMO norm of $f$, and such that $V$ converges in some non-tangential sense to $f$ almost everywhere with respect to $\sigma$. Our results should be compared to recent geometric characterizations of $L^p$-solvability and of BMO-solvability of the Dirichlet problem, by Azzam, the first author, Martell, Mourgoglou and Tolsa and by the first author and Le, respectively. In combination, this latter pair of results shows that one can construct, for all $f\in C_c(\pom)$, a {\em harmonic} extension $u$, with $|\nabla u(Y)|^2 \dist(Y,\pom) \, dY $ a Carleson measure with Carleson norm controlled by the BMO norm of $f$, \emph{only} in the presence of an appropriate quantitative connectivity condition.
\end{abstract}

\maketitle

\tableofcontents

\section*{List of symbols}

\begin{tabular}{cl}
  $C_\mu$               & Carleson norm of the measure $\mu$ (Definition \ref{defin:carleson_measure_constant}) \\
  $\Cs_\Ac$             & Carleson packing norm of $\Ac \subset \D$ (Definition \ref{defin:carleson_packing_norm}) \\
  $\D$                  & collection of dyadic cubes (Theorem \ref{theorem:existence_of_dyadic_cubes}) \\
  $\Gamma(x)$           & dyadic cone at $x \in \pom$ (Definition \ref{defin:dyadic_regions_and_cones}) \\
  $\widetilde{\Gamma}(x)$   & cone at $x \in \pom$ (Definition \ref{defin:cones}) \\
  $\Upsilon_Q(x)$       & semi-closed truncated cone at $x \in Q \subset \pom$ 
  (Section \ref{section:bilateral_corona}) \\
  $\widetilde{\Upsilon}_Q(x)$       & interior of $\Upsilon_Q(x)$ \\
  $\Omega$              & open set in $\R^{n+1}$ with ADR boundary $\pom$ \\
  $\omega^X$            & harmonic measure with pole at $X \in \Omega$ \\
  $\Uc_Q, U_Q$          & dilated and non-dilated closed Whitney 
  region (Sections \ref{section:bilateral_corona} and  \ref{section:carleson_boxes}) \\
  $U_Q^r$               & closed restricted Whitney region (Section \ref{section:carleson_regions}) \\
  $\Tc_Q, T_Q$          & semi-closed and open Carleson box (Sections \ref{section:bilateral_corona}
  and \ref{section:carleson_boxes})\\ 
  $\tau_Q$              & Carleson tent (Definition \ref{defin:dyadic_regions_and_cones}) \\
  $t_Q$                 & modified Carleson tent (Section \ref{section:carleson_regions})\\
  $\Wc$                 & Whitney cubes in $\Omega$ (Sections \ref{section:bilateral_corona}
  and \ref{section:carleson_boxes}) \\
  $G_{Q_0}$             & counting function with respect to $Q_0 \in \D$ (Lemma \ref{lemma:two-sided_cones}) \\
  $\delta, \beta$       & distance and smooth distance function with respect to $\pom$ (Theorem \ref{theorem:regularized_distance})\\
  $\langle f \rangle_A, \fint_A f$ & integral average of $f$ over $A$ (Section \ref{section:notation}) \\
  $\mathbb{N}_0$        & the set of non-negative integers $\{0,1,2,3,...\}$
\end{tabular}

\section{Introduction}

Connections between boundary geometry and PDE estimates have been studied for a long time (see e.g.\ the seminal work of F. and M. Riesz \cite{rieszs}) but the work is still ongoing and active. In the last couple of years, a lot of progress has been made, particularly in domains with codimension $1$ Ahlfors--David regular (ADR) or uniformly rectifiable (UR) boundaries (see \cite{hofmann_survey} for a survey of some of these recent advances). In this article, we complement recent results related to geometric characterizations of solvability of Dirichlet problems, by showing that an extension property for BMO functions, first proved by Varopoulos in the half-space \cite{varopoulos1,varopoulos2}, remains true even in settings where harmonic extension of BMO boundary data (i.e., BMO-solvability of the Dirichlet problem) may fail: in fact, we show in the present paper that the Varopoulos extension property holds always for UR sets of codimension 1.  In particular, our results do not require any kind of connectivity hypothesis on the domain or its boundary, whereas the analogous PDE solvability results cannot hold without certain 
quantitative connectivity assumptions.

Let us be more precise. Recently, Azzam, the first author, Martell, Mourgoglou and Tolsa \cite{azzametal1} have presented a geometric characterization of quantitative scale-invariant absolute 
continuity (i.e. the weak-$A_\infty$ property) of harmonic measure with respect to the surface measure. Their result together with recent work of the first author and Le \cite{hofmannle} gives us the following characterization theorem. For definitions of the properties mentioned in the theorem and in the rest of the introduction, see Section \ref{section:notation}.

\begin{theorem*}[{\cite{azzametal1,hofmannle}}]
  Let $\Omega \subset \R^{n+1}$ be an open set satisfying the corkscrew condition and suppose that $\partial \Omega$ is $n$-ADR. Then the following conditions are equivalent:
  \begin{enumerate}
    \item[(1)] $\partial \Omega$ is UR and $\Omega$ satisfies the weak local John condition,
    \item[(2)] harmonic measure belongs to the class weak-$A_\infty$ with respect to the surface measure 
    $\sigma \coloneqq \mathcal{H}^n\lfloor_{\pom}$ on $\partial \Omega$,
    \item[(3)] the Dirichlet problem is $L^p$-solvable for some $p < \infty$,
    \item[(4)] the Dirichlet problem is BMO-solvable.
  \end{enumerate}
\end{theorem*}
By $L^p$-solvability we mean that there exists a constant $C$ such that if $f \in L^p(\pom)$, then the solution $u$ to the Dirichlet problem with data $f$ converges non-tangentially to $f$ and
\begin{align*}
  \|N_* u\|_{L^p(\partial \Omega)} \le C \|f\|_{L^p(\partial \Omega)},
\end{align*}
where $N_*$ is a non-tangential maximal operator. Many key results related to this concept can be found in the monograph of Kenig \cite{kenig}. By BMO solvability\footnote{The definition is slightly different if $\Omega$ is unbounded and $\pom$ is bounded; see \cite[Section 5]{hofmannle} for details.}, we mean that there exists a constant $C$ such that if $f$ is a compactly supported continuous function on $\pom$, then the solution $u$ to the Dirichlet problem satisfies the Carleson measure estimate
\begin{align*}
  \sup_{x \in \partial \Omega, 0 < r \lesssim \diam(\partial \Omega)} \frac{1}{\sigma(\Delta(x,r))} \iint_{\Omega \cap B(x,r)} |\nabla u(Y)|^2 \delta(Y) \, dY \le C \|f\|_{\BMO(\partial \Omega)}^2,
\end{align*}
where $\Delta(x,r) \coloneqq B(x,r) \cap \pom$. This type of solvability was first shown to be equivalent to $L^p$-solvability, for some $p<\infty$, by Dindos, Kenig and Pipher \cite{dindoskenigpipher}, in Lipschitz or chord-arc domains (see also \cite{zhao} for an extension to 1-sided chord-arc domains).

It was previously known that the weak-$A_\infty$ property of harmonic measure (equivalently, $L^p$-solvability for some $p<\infty$) may fail in the absence of connectivity, even if the boundary is UR \cite{bishopjones}, but the result of \cite{azzametal1} is the first that tells us precisely how much connectivity we need (although we refer the reader to related work of Azzam \cite{azzam}, concerning the analogous geometric characterization problem, in the case that harmonic measure is doubling). In particular, there are many domains with ADR or even UR boundaries for which one does not have BMO-solvability, nor $L^p$-solvability for any finite $p$.

In this work, we nonetheless obtain extension results of Varopoulos type that can be seen as substitutes for these 
solvability theorems, in domains with $n$-UR boundaries, but in which the weak local John property may fail.
We first consider extensions of $L^\infty$ functions:

\begin{theorem}
  \label{theorem:bounded_extension}
Let $\Omega \subset \R^{n+1}$ be an open set satisfying the 
  corkscrew condition, with $n$-UR boundary. Then 
  for every Borel measurable 
$f \in L^\infty(\pom,d\sigma)$, there is
a function $\Phi = \Phi_f$ in $\Omega$, such that
  \begin{enumerate}
    \item[i)] $\Phi \in C^\infty(\Omega)$, and $|\nabla \Phi(X)|\leq C\|f\|_{L^\infty(\pom)} \,\delta(X)^{-1}$, for all 
    $X\in\Omega$.
    \item[ii)] $\|\Phi\|_{L^\infty(\Omega)} \le C \|f\|_{L^\infty(\pom)}$,
    \item[iii)] $\lim_{Y \to x \text{ N.T.}} \Phi(Y) = f(x)$ for $\sigma$-a.e.\ $x \in \pom$,
    \item[iv)] $|\nabla \Phi(Y)| \, dY$ is a Carleson measure:
                \begin{align*}
                  \sup_{r > 0, x \in \partial \Omega} \frac{1}{r^n} \iint_{B(x,r) \cap \Omega} |\nabla \Phi(Y)| \, dY \le C \|f\|_{L^\infty}.
                \end{align*}
  \end{enumerate}
Here, $\lim_{Y \to x \text{ N.T.}}$ stands for one-sided non-tangential 
 convergence\footnote{The notion of non-tangential convergence must be suitably interpreted in
  the present context.  We shall return to this matter in the sequel; see Definition \ref{defin:cones}, Lemma \ref{lemma:two-sided_traces}, and Remarks \ref{remark:cones}, \ref{r4.15} and \ref{r4.16}.}; $\sigma \coloneqq \Hc^n\lfloor_{\pom}$ is the surface measure,
  and $\delta(X) \coloneqq \dist(X,\pom)$ for $X\in \Omega$. The constant $C$ depends only on $n$, and the UR and corkscrew constants.
\end{theorem}

We remark that in particular, Theorem \ref{theorem:bounded_extension} applies in the case that $\Omega \coloneqq \ree \setminus E$, where $E$ is an arbitrary $n$-UR set: the corkscrew condition in that case is a simple (and well-known) consequence of Ahlfors--David regularity of $E$.

The proof is based on a combination of geometric arguments, potential theory and dyadic analysis, but the basic strategy follows that of Varopoulos \cite{varopoulos1,varopoulos2}: 
in particular, we strongly make use of the \emph{$\eps$-approximability} property of harmonic functions, established in the present context in \cite{hofmannmartellmayboroda}
(see Theorem \ref{theorem:eps-approximability} below). However, the implementation of this program is a delicate matter in the present generality, owing to the need to make harmonic extensions of functions belonging to $L^\infty(\pom,d\sigma)$, with non-tangential convergence $\sigma$-a.e.\ to the data, even though
harmonic measure may fail to be absolutely continuous with respect to surface measure $\sigma$; see Sections \ref{section:boundary} and \ref{section:proof_of_theorem_bounded},
and in particular Remark \ref{r6.4}.

Originally, the notion (although not the terminology) of $\eps$-approximability was introduced by Varopoulos \cite{varopoulos2}, and refined by Garnett \cite{garnett}, in order to study new ways to extend BMO functions inspired by Carleson's corona theorem \cite{carleson}, and the closely related topic of $H^1$--BMO duality (see particularly \cite[Theorem 3]{feffermanstein}). 
The $\eps$-approximability property provides a convenient detour to circumvent the unfortunate fact that there exist harmonic functions $u$ such that $|\nabla u(Y)| \, dY$ is not a Carleson measure \cite{garnett}. Subsequently, this property has offered ways to connect Carleson measure estimates for solutions, with quantitative Fatou Theorems  \cite{garnett}, \cite{bortzhofmann}, with absolute continuity properties of elliptic measures \cite{kenigkochpiphertoro,hofmannkenigmayborodapipher} and with boundary geometry \cite{hofmannmartellmayboroda,garnettmourgogloutolsa,azzamgarnettmourgogloutolsa,bortzhofmann,hofmanntapiola,bortztapiola}.

Our second result is the following generalization of \cite[Theorem 2]{varopoulos2}, which in some sense provides a substitute for BMO-solvability of the Dirichlet problem: 

\begin{theorem}
  \label{theorem:varopoulos_extension}
  Suppose that $\Omega \subset \R^{n+1}$ is an open set satisfying the corkscrew condition with $n$-UR boundary. Then there exists a constant $C$ such that if $f \in \BMO(\pom,d\sigma)$ is compactly supported, 
then there exists a function $V = V_f$ in $\Omega$ such that
  \begin{enumerate}
    \item[i)] $V \in C^\infty(\Omega)$, and $|\nabla V(X)|\leq C \|f\|_{BMO} \,\delta (X)^{-1}$, for all $X\in \Omega$,
    \item[ii)] $\lim_{Y \to x \text{ N.T.}} V(Y) = f(x)$ for $\sigma$-a.e.\ $x \in \pom$,
    \item[iii)] $|\nabla V(Y)| \, dY$ is a Carleson measure:
                \begin{align*}
                  \sup_{r > 0, x \in \partial \Omega} \frac{1}{r^n} \iint_{B(x,r) \cap \Omega} |\nabla V(Y)| \, dY \le C \|f\|_{\BMO}.
                \end{align*}
  \end{enumerate}
  Here $\lim_{Y \to x \text{ N.T.}}$ stands for one-sided non-tangential convergence (see Definition \ref{defin:cones} and Remark \ref{remark:cones}); $\delta(X) \coloneqq \dist(X,\pom)$, and $\sigma \coloneqq \Hc^n\lfloor_{\pom}$ is the surface measure.
\end{theorem}
The proof is a combination of Theorem \ref{theorem:bounded_extension}, Garnett's decomposition lemma (see Lemma \ref{lemma:garnett}), and the following extension result for the ``dyadic part'' of Garnett's lemma:
\begin{proposition}
  \label{proposition:dyadic_extension}
  Suppose that $\Omega \subset \R^{n+1}$ is an open set satisfying the corkscrew condition with $d$-ADR boundary for some $d \in (0,n]$.
  Let $\D$ be a dyadic system on $\pom$, $Q_0 \in \D$ be a fixed dyadic cube and $\{Q_j\}_j \subset \D_{Q_0}$ be a collection of subcubes of $Q_0$. Suppose that function $f$ in $\pom$, 
  $f(x) = \sum_j \alpha_j 1_{Q_j}$, satisfies the following conditions:
  \begin{enumerate}
    \item[$\bullet$] $f \in \BMO(\pom)$,
    \item[$\bullet$] there exists $C_0 \ge 1$ such that $\sum_{Q_j \subset Q} \sigma(Q_j) \le C_0 \sigma(Q)$ for every $Q \in \D$,
    \item[$\bullet$] $\sup_j |\alpha_j| \lesssim \|f\|_{\BMO}$.
  \end{enumerate}
  Then there exists a function $F = F_f$ in $\Omega$ such that
  \begin{enumerate}
    \item[i)] $F \in C^\infty(\Omega)$, and $|\nabla F(X)|\leq C \|f\|_{BMO} \,\delta (X)^{-1}$, for all $X\in \Omega$,
    \item[ii)] $\lim_{Y \to x \text{ N.T.}} F(Y) = f(x)$ for $\sigma$-a.e.\ $x \in \pom$,
    \item[iii)] $|\nabla F(Y)| \, dY$ satisfies a quantitative codimension $1$ type Carleson measure estimate:
                \begin{align}\label{eq1.4}
                  \sup_{r > 0, x \in \partial \Omega} \frac{1}{r^n} \iint_{B(x,r) \cap \Omega} |\nabla F(Y)| \, dY \lesssim C_0 \|f\|_{\BMO}.
                \end{align}
  \end{enumerate}
  Here $\lim_{Y \to x \text{ N.T.}}$ stands for standard 
  type non-tangential convergence (see Definition \ref{defin:cones}); $\delta(X) \coloneqq \dist(X,\pom)$,  
  and $\sigma \coloneqq \Hc^d\lfloor_{\pom}$ is the surface measure.
\end{proposition}
We remark that in proving Theorem \ref{theorem:varopoulos_extension}, we shall use only the codimension 1 case (i.e., $d=n$) of Proposition \ref{proposition:dyadic_extension}.

Unlike that of Theorem \ref{theorem:bounded_extension}, the proof of Proposition \ref{proposition:dyadic_extension} does not require any UR machinery. Many of the key arguments are fairly elementary but still a bit delicate. A principal difficulty is the need to build suitable substitutes for Carleson boxes that are compatible with non-tangential convergence, as well as with proving
the Carleson measure estimate \eqref{eq1.4}.  Both the construction of our boxes and the rest of our techniques work for $d$-ADR boundaries for any $d \in (0,n]$, including non-integer dimensions.

We conjecture that if $\Omega \subset \R^{n+1}$ is an open set satisfying the corkscrew condition with $n$-ADR boundary $\pom$, then the existence of extensions (with some suitable convergence to the boundary values) as in Theorem \ref{theorem:bounded_extension} implies that $\pom$ is $n$-UR. We note that if these extensions exist, then also extensions as in Theorem \ref{theorem:varopoulos_extension} exist since Proposition \ref{proposition:dyadic_extension} and Lemma \ref{lemma:garnett} hold with just the ADR assumption. 

The paper is organized as follows. In the next section, we discuss the basic notation and definitions in the paper. In Section \ref{section:eps-approximability}, we consider $\eps$-approximators and many regularization lemmas we need later. We build machinery for Theorem \ref{theorem:bounded_extension} in Sections \ref{section:bilateral_corona} and  \ref{section:boundary},
and we prove the theorem in Section \ref{section:proof_of_theorem_bounded}. In Sections \ref{section:carleson_boxes} and \ref{section:carleson_regions}, we revisit and modify the construction of Whitney regions and Carleson boxes and we use the modified construction to prove Proposition \ref{proposition:dyadic_extension} in Section \ref{section:dyadic_extension}. Finally, in Section \ref{section:garnett}, we prove a version of Garnett's decomposition lemma and combine it with Theorem \ref{theorem:bounded_extension} and Proposition \ref{proposition:dyadic_extension} to prove Theorem \ref{theorem:varopoulos_extension}.

\section{Notation and basic definitions}
\label{section:notation}

We use the following notation.
\begin{enumerate}
  \item[$\bullet$] $\Omega \subset \R^{n+1}$ will always be an open set with non-empty $d$-dimensional ADR boundary $\partial \Omega$ (see Definition \ref{defin:adr}).
                   In Sections \ref{section:bilateral_corona}, 
   \ref{section:boundary}, and  \ref{section:proof_of_theorem_bounded}, we additionally assume that $\partial \Omega$ is 
                   $n$-UR (see Definition \ref{defin:ur}) and that $\Omega$ satisfies the corkscrew condition (see Definition \ref{defin:corkscrew}).
  
  \item[$\bullet$] The letters $c$ and $C$ denote constants that depend only on dimension, ADR constant (see Definition 
                   \ref{defin:adr}), UR constants (see Definition \ref{defin:ur}) and other similar parameters. The values of $c$ and $C$ may change from one occurence to another. We do not track 
                   how our bounds depend on these constants and usually just write $\gamma_1 \lesssim \gamma_2$ if $\gamma_1 \le c\gamma_2$ for a constant like this $c$ and $\gamma_1 \approx \gamma_2$ if $\gamma_1 \lesssim \gamma_2 \lesssim \gamma_1$. If the constant $c_\kappa$ depends only on parameters of the previous type and some other parameter $\kappa$, we usually write $\gamma_1 \lesssim_\kappa \gamma_2$ instead of $\gamma_1 \le c_\kappa \gamma_2$.
                   
  \item[$\bullet$] We use capital letters $X,Y,Z$, and so on to denote points in $\Omega$
                   and lowecase letters $x,y,z$, and so on to denote points in $\partial \Omega$.
                   
  \item[$\bullet$] The $(n+1)$-dimensional Euclidean open ball of radius $r$ will be denoted $B(x,r)$ or $B(X,r)$ depending 
                   on whether the center point lies on $\pom$ or $\Omega$. We denote the surface ball of radius $r$ centered
                   at $x$ by $\Delta(x,r) \coloneqq B(x,r) \cap \pom$.
                   
  \item[$\bullet$] Given a Euclidean ball $B \coloneqq B(X,r)$ or a surface ball $\Delta \coloneqq \Delta(x,r)$ and constant 
                   $\kappa > 0$, we denote $\kappa B \coloneqq B(X,\kappa r)$ and $\kappa \Delta \coloneqq \Delta(x, \kappa r)$.
                   
  \item[$\bullet$] For every $X \in \Omega$ we set $\delta(X) \coloneqq \dist(X,\pom)$.
  
  \item[$\bullet$] We let $\Hc^d$ be the $d$-dimensional Hausdorff measure and denote the surface measure of $\pom$ by $\sigma \coloneqq \Hc^d\lfloor_{\partial \Omega}$. The $(n+1)$-dimensional 
                   Lebesgue measure of a measurable set $A \subset \Omega$ will be denoted by $|A|$.
  
  \item[$\bullet$] For a set $A \subset \R^{n+1}$, we let $1_A$ be the indicator function of $A$: $1_A(x) = 0$ if $x \notin A$
                   and $1_A(x) = 1$ if $x \in A$.
                   
  \item[$\bullet$] The interior of a set $A$ will be denoted $\text{int}(A)$.
  
  \item[$\bullet$] The unit outer normal (when it exists) will be denoted by $\overrightarrow{N}$.
  
  \item[$\bullet$] For $\mu$-measurable sets $A$ with positive and finite measure we set $\langle f \rangle_A \coloneqq \fint_A f \, d\mu \coloneqq \tfrac{1}{\mu(A)} f \, d\mu$.
\end{enumerate}

\begin{defin}[Cones and non-tangential limits]
  \label{defin:cones}
  Suppose that $m > 1$. For every $x \in \partial \Omega$, the \emph{cone of $m$-aperture at $x$} is the set
  \begin{align*}
    \widetilde{\Gamma}(x) \coloneqq \widetilde{\Gamma}^m(x) \coloneqq \{Z \in \Omega \colon \dist(Z,x) < m\delta(Z)\}.
  \end{align*}
  Let $G$ be a function defined in $\Omega$, $g$ be a function defined on $\pom$ and $x$ be a point on $\pom$. We consider two types of non-tangential convergence in this paper. We use the notation $\lim_{Y \to x \, \text{N.T.}} G(Y) = g(x)$ for both of them, 
  but the meaning 
  should be clear from context.
  \begin{enumerate}
    \item[$\bullet$] With \emph{standard type non-tangential convergence} we mean that there exists $m > 1$ such that we have $\lim_{k \to \infty} G(Y_k) = g(x)$ for every sequence $(Y_k)$ in $\widetilde{\Gamma}^m(x)$ such that $\lim_{k \to \infty} Y_k  = x$.
    
    \item[$\bullet$] With \emph{one-sided non-tangential convergence} we mean that there exists $m > 1$ and a connected component $A \subset \widetilde{\Gamma}^m(x)$ such that $x \in \partial A$ and $\lim_{k \to \infty} G(Y_k) = g(x)$ for every sequence $(Y_k)$ in $A$ such that $\lim_{k \to \infty} Y_k  = x$.
  \end{enumerate}
\end{defin}

\begin{remark}
  \label{remark:cones}
  \
  \begin{enumerate}
    \item[i)] If $\Omega \subset \R^{n+1}$ is an open set satisfying the corkscrew condition, with 
    UR boundary $\pom$, then for $\sigma$-a.e. $x\in \pom$, the cone with vertex at $x$ 
    has at most two connected components inside $\Omega$ such that their boundaries contain $x$, 
    by Lemma \ref{lemma:two-sided_cones} 
    (see also Lemma \ref{lemma:two-sided_traces}, and Remarks \ref{r4.15} and \ref{r4.16}).
  
    \item[ii)] In the actual calculations related to non-tangential convergence, we use \emph{dyadic cones} that we define in later sections (see Section \ref{section:bilateral_corona} and Section \ref{section:carleson_boxes}). These dyadic cones always contain a truncated cone of the type $\widetilde{\Gamma}(x)$, at least locally.
  \end{enumerate}
\end{remark}

\begin{defin}[BMO and dyadic BMO]\label{bmodef}
  The space $\BMO(\pom)$ (\emph{bounded mean oscillation}) consists of those locally integrable function $f$ such that
  \begin{align*}
    \|f\|_\BMO \coloneqq \sup_\Delta \fint_\Delta |f(y) - \langle f \rangle_\Delta| \, d\sigma(y) < \infty,
  \end{align*}
  where the supremum is taken over all surface balls $\Delta \subset \partial \Omega$. We define the \emph{dyadic BMO space}
  $\BMO_\D(\partial \Omega)$ by replacing the supremum over all surface balls with the supremum over all dyadic cubes $Q$ 
  (see Theorem \ref{theorem:existence_of_dyadic_cubes}).
\end{defin}

\begin{defin}[Carleson measures]
  \label{defin:carleson_measure_constant}
  We say that a Borel measure $\mu$ in $\Omega$ is a \emph{Carleson measure (with respect to $\partial \Omega$)} if we have
  \begin{align}
    \label{constant:carleson_measure_constant}
    C_\mu \coloneqq \sup_{x \in \pom, r > 0} \frac{\mu(B(x,r) \cap \Omega)}{r^n} < \infty.
  \end{align}
  We call $C_\mu$ the \emph{Carleson norm of $\mu$}.
\end{defin}

\begin{defin}[Local BV]
  We say that locally integrable function $f$ has \emph{locally bounded variation in $\Omega$} (denote $f \in \text{BV}_{\text{loc}}(\Omega)$)
  if for any open relatively compact set $\Omega' \subset \Omega$ the \emph{total variation over $\Omega'$} is finite:
  \begin{align*}
    \iint_{\Omega'} |\nabla \varphi| \, dY 
    \coloneqq \sup_{\substack{\overrightarrow{\Psi} \in C_0^1(\Omega') \\ \|\overrightarrow{\Psi}\|_{L^\infty(\Omega') \le 1}}} \iint_{\Omega'} \varphi \, \text{div} \overrightarrow{\Psi} \, dY < \infty,
  \end{align*}
  where $C_0^1(\Omega')$ is the class of compactly supported continuously differentiable vector fields in $\Omega'$.
\end{defin}

\begin{defin}[Carrot paths]
  Let $X \in \Omega$ and $y \in \pom$. A connected rectifiable path $\gamma$ from $X$ to $y$ is a \emph{$\lambda$-carrot path} if $\gamma \setminus \{y\} \subset \Omega$ and for every $Z \in \gamma$ we have $\lambda \ell(\gamma(y,Z)) \le \delta(Z)$.
\end{defin}

\begin{defin}[Weak local John condition]
  We say that $\Omega$ satisfies the \emph{weak local John condition} if there exist constants $\lambda \in (0,1)$, $\theta \in (0,1]$ and $R \ge 2$ such that for every $X$ there exists a Borel set $F \subset \Delta_X \coloneqq B(X,R\delta(X)) \cap \pom$ such that $\sigma(F) \ge \theta \sigma(\Delta_X)$ and for every $y \in F$ there is a $\lambda$-carrot path connecting $y$ to $X$.
\end{defin}

\begin{defin}[Weak $A_\infty$]
  Let $\nu$ be a measure defined on $\pom$ and $\Delta_0 \coloneqq B_0 \cap \pom$ be a surface ball. We say that $\nu$ belongs to weak-$A_\infty(\Delta_0)$ if there are positive constants $C$ and $s$ such that for each surface ball $\Delta \coloneqq B \cap \pom$ centered on $\pom$ with $B \subset B_0$ we have
  \begin{align}
    \label{condition:weak_a_infty}
    \nu(A) \le C \left( \frac{\sigma(A)}{\sigma(\Delta)} \right)^s \nu(2\Delta)
  \end{align}
  for every Borel set $A \subset \Delta$.
\end{defin}

We note that the constant $2$ in \eqref{condition:weak_a_infty} can be replaced with any constant $c > 1$ without changing the class weak-$A_\infty(\Delta_0)$ (see e.g. \cite[Section 8]{andersonhytonentapiola}).

\subsection{ADR, UR, NTA, CAD, and corkscrew condition}

\begin{defin}[ADR]
  \label{defin:adr}
  We say that a closed set $E \subset \R^{n+1}$ is a \emph{$d$-ADR (Ahlfors--David regular)} set for $d \in (0,n]$ if there exists a uniform constant $C$ such that
  \begin{align*}
    \frac{1}{C} r^d \le \sigma(\Delta(x,r)) \le C r^d
  \end{align*}
  for every $x \in E$ and every $r \in (0,\diam(E))$, where $\diam(E)$ may be infinite.
\end{defin}

\begin{defin}[Corkscrew condition]
  \label{defin:corkscrew}
  We say that $\Omega$ satisfies the \emph{corkscrew condition} if there exists a uniform constant $c$ such that for every surface ball $\Delta \coloneqq \Delta(x,r)$ with $x \in \partial\Omega$ and $0 < r < \diam(\partial \Omega)$ there exists a point $X_\Delta \in \Omega$ such that $B(X_\Delta, cr) \subset B(x,r) \cap \Omega$,
\end{defin}

\begin{defin}[UR]
  \label{defin:ur}
  Following \cite{davidsemmes_singular, davidsemmes_analysis}, we say that an $n$-ADR set $E \subset \R^{n+1}$ is \emph{UR (uniformly rectifiable)} if it contains ``big pieces of Lipschitz images'' (BPLI) of $\R^n$: there exist constants $\theta, M > 0$ such that for every $x \in E$ and $r \in (0,\diam(E))$ there is a Lipschitz mapping
  $\rho = \rho_{x,r} \colon \R^n \to \R^{n+1}$, with Lipschitz norm no larger that $M$, such that
  \begin{align*}
    \Hc^n(E \cap B(x,r) \cap \rho(\{y \in \R^n \colon |y| < r\})) \ge \theta r^n.
  \end{align*}
\end{defin}
As it is well-known, UR is a necessary and sufficient condition for many types of PDE and Calder\'on--Zygmund type harmonic analysis results on ADR sets or open sets with ADR boundaries. In this paper, we work with two characterizations UR: $\eps$-approximability of harmonic function (see Section \ref{section:eps-approximability}) and bilateral corona decomposition (see Section \ref{section:bilateral_corona}). We use $\eps$-approximability to build the extension in Theorem \ref{theorem:bounded_extension}, and the bilateral corona
decomposition, and its consequences, as a tool to prove some convergence properties.

\begin{defin}[NTA]
  Following \cite{jerisonkenig}, we say that a domain $\Theta \subset \R^{n+1}$ is \emph{NTA (nontangentially accessible)} if
  \begin{enumerate}
    \item[$\bullet$] $\Theta$ satisfies the \emph{Harnack chain condition}: there exists a 
                     uniform constant $C$ 
                     such that for every $\rho > 0$, $\Lambda \ge 1$ and $X,X' \in \Theta$ with $\delta(X), \delta(X') \ge \rho$
                     and $|X - X'| < \Lambda \rho$ there exists a chain of open balls $B_1, \ldots, B_N \subset \Theta$, $N \le C(\Lambda)$,
                     with $X \in B_1$, $X' \in B_N$, $B_k \cap B_{k+1} \neq \emptyset$ and $C^{-1} \diam(B_k) \le \dist(B_k,\partial \Theta) \le C \diam(B_k)$,
    
    \item[$\bullet$] both $\Theta$ and $\R^{n+1} \setminus \Theta$ satisfy the corkscrew condition.
  \end{enumerate}
\end{defin}

\begin{defin}[CAD] \label{CAD}
  An open set $\Omega\subset \ree$ is a {\em CAD (chord-arc domain)} if it is NTA, and $\pom$ is n-ADR. The constants in the Harnack chain, corkscrew, and ADR conditions are referred to collectively as the \emph{chord-arc constants}.
\end{defin}

\subsubsection{Dyadic cubes}
\label{section:dyadic_cubes}
\begin{theorem}[E.g. {\cite{christ, sawyerwheeden, hytonenkairema}}]
  \label{theorem:existence_of_dyadic_cubes}
  Suppose that $E$ is a $d$-ADR set. Then there exists a countable collection $\D$ (that we call a \emph{dyadic system}),
  \begin{align*}
    \D \coloneqq \bigcup_{k \in \Z} \D_k, \ \ \ \ \ \D_k \coloneqq \{ Q_\alpha^k \colon \alpha \in \mathcal{A}_k \}
  \end{align*}
  of Borel sets $Q_\alpha^k$ (that we call \emph{dyadic cubes}) such that
  \begin{enumerate}
    \item[(i)] the collection $\D$ is \emph{nested}: $\text{if } Q,P \in \D, \text{ then } Q \cap P \in \{\emptyset,Q,P\}$,
    \item[(ii)] $E = \bigcup_{Q \in \D_k} Q$ for every $k \in \Z$ and the union is disjoint,
    \item[(iii)] there exist constants $c_1 > 0$ and $C_1 \ge 1$ such that
                     \begin{align}
                       \label{dyadic_cubes:balls_inclusion} \Delta(z_\alpha^k,c_1 2^{-k}) \subseteq Q_\alpha^k \subseteq \Delta(z_\alpha^k, C_1 2^{-k}) \eqqcolon \Delta_{Q_\alpha^k},
                     \end{align}
    \item[(iv)] for every set $Q_\alpha^k$ there exists at most $N$ cubes $Q_{\beta_i}^{k+1}$ (called the \emph{children} of $Q_\alpha^k$) such that $Q_\alpha^k = \bigcup_i Q_{\beta_i}^{k+1}$, 
               where the constant $N$ depends only on the ADR constant of $E$,
    \item[(v)] the cubes have \emph{thin boundaries}: there exists a constant $\gamma > 0$ such that
                \begin{align}
                  \label{dyadic_cubes:small_boundaries} \sigma\left( \left\{x \in Q_\alpha^k \colon \dist\left(x, E \setminus Q_\alpha^k \right) \le \varrho 2^{-k} \right\} \right)
                  \le C_1 \varrho^\gamma \sigma(Q_\alpha^k)
                \end{align}
                for all cubes $Q_\alpha^k$ and for all $\varrho \in (0,c_1)$.
  \end{enumerate}
  In addition, there exists a collection of dyadic systems $\{\D^\nu\}_{\nu=1}^N$ on $E$, of bounded cardinality $N$, and a uniform constant $C$, such that if  $\Delta = B \cap E$ is
  any surface ball centered on $E$, then there is at least one choice of dyadic system $\D^\nu$, and a cube $Q \in \D^\nu$, with $\Delta \subset Q$, and with $\diam(Q) \leq \min(C\diam(B),\diam(E))$.
\end{theorem}

\begin{remark}
  In general spaces of homogeneous type, dyadic systems were first constructed in \cite{christ} for some parameter $\delta \in (0,1)$ instead of the dyadic parameter $1/2$ (we may always choose $\delta = 1/2$ by \cite{hofmannmitreamitreamorris}). In the same context, the \emph{adjacent systems} $\{\D^\nu\}_{\nu=1}^N$ were contructed in \cite{hytonenkairema} (see also \cite{hytonentapiola,tapiola} for an alternative construction and some additional approximation properties in geometrically doubling metric spaces). For the history of adjacent systems in $\R^n$, see \cite[Section 3]{cruz-uribe}.
\end{remark}

\begin{notation}
  \label{notation:dyadic_cubes} We shall use the following notational conventions.
  \begin{enumerate}
    \item[(1)] Since the boundary $\partial \Omega$ may be bounded or disconnected, we may encounter a situation where 
              $Q_\alpha^k = Q_\beta^l$ although $k \neq l$. Thus, when we consider cubes $Q_\alpha^k \in \D$, we assume that
              $C_1 2^{-k} \le \diam(\pom)$ and the number $k$ is maximal in the sense that there does not exist a cube $Q_\beta^l \in \D$ such that 
              $Q_\beta^l = Q_\alpha^k$ for some $l > k$. Notice that the number $k$ is bounded for each cube since the ADR condition excludes the 
              presence of isolated points in $\partial \Omega$.
              
    \item[(2)] For each $k$, and for every cube $Q_\alpha^k \coloneqq Q \in \D_k$, we denote $\ell(Q) \coloneqq 2^{-k}$ and $x_Q \coloneqq z_\alpha^k$. We call
              $\ell(Q)$ the \emph{side length} of $Q$, and $x_Q$ the \emph{center} of $Q$.
              
    \item[(3)] For every $Q \in \D$, we denote the collection of dyadic subcubes of $Q$ by $\D_Q$.
    
    \item[(4)] For every $Q \in \D$ and $\kappa > 0$, we denote $\kappa Q \coloneqq \kappa \Delta_Q$.
  \end{enumerate}
\end{notation}

\begin{remark} We record the following further observations.
  \begin{enumerate}
   \item[(1)]  The following exterior variant of \eqref{dyadic_cubes:small_boundaries}  in Theorem
   \ref{theorem:existence_of_dyadic_cubes} also holds for every $Q\in \D$:
    \begin{align}
      \label{dyadic_cubes:exterior_small_boundaries} 
      \sigma\big(\text{Ext}_\varrho(Q)\big) \coloneqq \sigma\big( \left\{x \in E \setminus Q \colon \dist\left(x, Q \right) \le \varrho \ell(Q) \right\}\big)
      \lesssim C_1 \varrho^\gamma \sigma(Q)\,,
    \end{align}
  as may be seen by covering the \emph{exterior shell} $\text{Ext}_\varrho(Q)$ by dyadic cubes of uniform side length $\approx \varrho\ell(Q)$, each of which is a subcube of one of a uniformly bounded number of neighbors of $Q$ with side length equal to that of $Q$. Applying \eqref{dyadic_cubes:small_boundaries} in each of these neighbors,
  we obtain \eqref{dyadic_cubes:exterior_small_boundaries}.

  \item[(2)] By the ADR property and \eqref{dyadic_cubes:balls_inclusion}, we have $\sigma(Q) \approx \ell(Q)^d$ with implicit constants 
               independent of $Q$, and $\sigma(\widetilde{Q}) \lesssim \sigma(Q)$ for the \emph{dyadic parent of $Q$}, 
that is, the cube $\widetilde{Q}$ containing $Q$, and belonging to the generation immediately preceeding
that of $Q$,
(i.e., $\widetilde{Q} \in \D_{k-1}$ when $Q\in \D_k$).
               Similarly we have $\sigma(\kappa Q) \lesssim_\kappa \sigma(Q)$ for all $\kappa > 1$.
  \end{enumerate}
\end{remark}

\begin{defin}
  \label{defin:carleson_packing_norm}
  We say that a collection $\Ac \subset \D$ satisfies a \emph{Carleson packing condition} if there exists a 
  constant $C \ge 1$ such that
  \begin{align*}
    \sum_{Q \in \Ac, Q \subset Q_0} \sigma(Q) \le C \sigma(Q_0)
  \end{align*}
  for every cube $Q_0 \in \D$. We call the smallest such constant $C$ the \emph{Carleson packing norm of $\Ac$} and denote it 
  by $\Cs_{\Ac}$.
\end{defin}

\begin{lemma}
  \label{lemma:carleson_sums}
  Suppose $E\subset \ree$ is a $d$-ADR set and that $\Ac \subset \D$ satisfies a Carleson packing condition. Then we have
  \begin{align*}
    \sum_{Q \in \Ac, Q \subset Q_0} \ell(Q)^n \lesssim \Cs_{\Ac} \ell(Q_0)^n
  \end{align*}
  for every cube $Q \in \D$ and every $d\leq n$.
\end{lemma}

\begin{proof} For $d=n$, in the presence of the $n$-ADR condition, the lemma is a trivial reformulation of Definition \ref{defin:carleson_packing_norm}.  
Therefore let us suppose that $d<n$. In this case, the same trivial argument using $d$-ADR gives
$\sum_{Q \in \Ac, Q \subset Q_0} \ell(Q)^d \lesssim \Cs_{\Ac} \ell(Q_0)^d$. Consequently,
\begin{multline*}
    \sum_{Q \in \Ac, Q \subset Q_0} \ell(Q)^n =
    \sum_{Q \in \Ac, Q \subset Q_0} \ell(Q)^{n-d} \ell(Q)^d
    \leq  \ell(Q_0)^{n-d}   \sum_{Q \in \Ac, Q \subset Q_0}  \ell(Q)^d \\[4pt]
    \lesssim  \Cs_{\Ac}  \ell(Q_0)^{n-d}   \ell(Q_0)^d =  \Cs_{\Ac}  \ell(Q_0)^{n}.
\end{multline*}
\end{proof}

\section{$\eps$-approximability and regularization}
\label{section:eps-approximability}

In the proof of Theorem \ref{theorem:bounded_extension}, we follow the original idea of Varopoulos and construct the extension using $\eps$-approximability of harmonic functions. It was recently shown that this property characterizes uniform rectifiability:

\begin{theorem}[{\cite{hofmannmartellmayboroda, garnettmourgogloutolsa}}]
  \label{theorem:eps-approximability}
  Suppose that $\Omega \subset \R^{n+1}$ is an open set satisfying the corkscrew condition. 
  Then $\partial \Omega$ is UR 
  if and only if every bounded harmonic function $u$ in $\Omega$ is \emph{$\eps$-approximable} for 
  every $\eps \in (0,1)$:
  there exists a constant $C_\eps$ and a function $\Phi = \Phi^\eps \in \text{BV}_{\text{loc}}(\Omega)$ such that
  \begin{align*}
    \|u - \Phi\|_{L^\infty} \le \eps\|u\|_{L^\infty(\Omega)} \ \ \text{ and } \ \ \sup_{x \in \pom, r > 0}  \frac{1}{r^n} \iint_{B(x,r) \cap \Omega} |\nabla \Phi(Y)| \, dY \le C_\eps \|u\|_{L^\infty},
  \end{align*}
  i.e. $|\nabla \Phi(Y)| \, dY$ is a Carleson measure.
\end{theorem}
The direction UR implies $\eps$-approximability appears in \cite{hofmannmartellmayboroda}, and the converse is proved in \cite{garnettmourgogloutolsa}. 
(see also \cite{hofmanntapiola} and \cite{bortztapiola} for pointwise and $L^p$ versions of this result). 
For other characterizations of UR with respect to properties of harmonic functions or solutions to other elliptic PDE, see \cite{hofmannmartellmayboroda,hofmannmartellmayboroda_square,garnettmourgogloutolsa,hofmanntapiola,bortztapiola,azzamgarnettmourgogloutolsa}.

Since $\eps$-approximators are a crucial ingredient in the proof of Theorem \ref{theorem:bounded_extension}, it will be convenient for us to use regularized $\eps$-approximators that are locally Lipschitz:
\begin{lemma}
  \label{lemma:smooth_approximators}
  We can choose the $\eps$-approximators $\Phi = \Phi^\eps$ in Theorem 
  \ref{theorem:eps-approximability} so that
  \begin{enumerate}
    \item[i)] $\Phi \in C^\infty(\Omega)$,
    \item[ii)] $|\nabla \Phi(Y)| \lesssim \tfrac{1}{\delta(Y)}$ for every $Y \in \Omega$,
    \item[iii)] if $|X-Y| \ll \delta(X)$, then $|\Phi(X) - \Phi(Y)| \lesssim \tfrac{|X-Y|}{\delta(X)}$.
  \end{enumerate}
\end{lemma}
We shall verify this lemma by a fairly straightforward mollifier argument (see e.g. \cite[Section 4]{evansgariepy})). Since we need to regularize also other functions in subsequent sections, we formulate the following lemmas in a fairly general way.

We start by noting that although our distance function $\delta$ is Lipschitz, that is usually the best level of regularity we can hope for in this context. However, we can use a classical result of Stein to replace $\delta$ with a smooth function that is pointwise close to $\delta$:
\begin{theorem}[{\cite[Theorem 2, p. 171]{stein}}]
  \label{theorem:regularized_distance}
  Let $E \subset \R^{n+1}$ be a closed set and $\delta_E$ be the distance function with respect to $E$. Then there exist positive constants $m_1$ and $m_2$ and a function $\beta_E$ defined in $E^c$ such that
  \begin{enumerate}
    \item[(i)] $m_1 \delta_E(x) \le \beta_E(x) \le m_2 \delta_E(x)$ for every $x \in E^c$, and
    \item[(ii)] $\beta_E$ is smooth in $E^c$ and
                \begin{align*}
                  \left| \frac{\partial^\alpha}{\partial x^\alpha} \beta_E(x) \right| \le C_\alpha \beta_E(x)^{1-|\alpha|}.
                \end{align*}
  \end{enumerate}
  In addition, the constants $m_1$, $m_2$ and $C_\alpha$ are independent of $E$.
\end{theorem}
For a given closed set $E\subset \ree$, let $\delta \coloneqq \delta_E$, $\beta \coloneqq \beta_E$,  and $m_2>0$ be as in Theorem \ref{theorem:regularized_distance}. Let $\zeta \ge 0$ be a smooth non-negative function supported on $B(0,\tfrac{1}{2m_2})$, satisfying $\zeta \le 1$ and $\int \zeta = 1$.  For every $\lambda > 0$, we set
\begin{align*}
  \zeta_\lambda(X) \coloneqq \frac{1}{\lambda^{n+1}} \zeta\left( \frac{X}{\lambda} \right)\,,
\end{align*}
and define
  \begin{align*}
    \Lambda(X,Y) =  \zeta_{\beta(X)}(X-Y)=\frac{1}{\beta(X)^{n+1}} \zeta\left( \frac{X-Y}{\beta(X)} \right).
  \end{align*}
  Set $\Omega \coloneqq \ree\setminus E$, so that $\pom=E$, and
  observe that for given $X\in \Omega$, 
  \begin{equation}\label{eq3.4}
   \text{supp } \Lambda(X,\cdot) \subset B_X \coloneqq B(X,\delta(X)/2)\,,
   \end{equation}
   by construction. 
  Suppose that $G_0 \colon \Omega \to \R$ is a locally integrable function. We set
\begin{align}
  \label{defin:smooth_G}
  G(X) \coloneqq \iint \Lambda(X,Y) G_0(Y) \, dY\,.
\end{align}
We then have the following.
\begin{lemma}
  \label{lemma:smooth_regularization}
  $G \in C^\infty(\Omega)$ and
  \begin{align}
    \label{identity:gradient}
    \nabla G(X) = \iint \nabla_X \Lambda(X,Y) G_0(Y) \, dY.
  \end{align}
\end{lemma}
The proof is a routine modification of the case $\Omega = \R^{n+1}$ (see e.g. the proof of \cite[Theorem 1 (i), p.\ 123]{evansgariepy}).
\begin{lemma}
  \label{lemma:gradient_regularization}
  If $G_0 \in BV_{\text{loc}}(\Omega)$ and $\mu = |\nabla G_0(Y)| \, dY$ is a Carleson measure, then
  \begin{align}
    \label{estimate:gradient_pointwise} |\nabla G(X)| \lesssim \frac{C_\mu}{\delta(X)}
  \end{align}
  for every $X \in \Omega$, where $C_\mu$ is the constant in \eqref{constant:carleson_measure_constant}.
\end{lemma}

\begin{proof}
We begin with some preliminary observations. With $B_X$ defined as in \eqref{eq3.4}, note that by Theorem \ref{theorem:regularized_distance} and construction,
\begin{equation} \label{eq3.11a}
  \sup_{Y\in B_X}| \Lambda(X,Y)| \lesssim \delta(X)^{-n-1}\,,
  \quad
  \sup_{Y\in B_X}| \nabla_X \Lambda(X,Y) | \lesssim \delta(X)^{-n-2}\,.
\end{equation}
Moreover,  $\iint \Lambda(X,Y) \, dY = 1$, for every $X\in \Omega$, and therefore
\begin{align}
  \label{identity:gradient_zero}
  \nabla_X \iint \Lambda(X,Y) \, dY \overset{\eqref{identity:gradient}}{=} \iint \nabla_X \Lambda(X,Y) \, dY = 0\,.
\end{align}
Set $[G_0]_{B_X} \coloneqq |B_X|^{-1} \iint_{B_X} G_0$.  Then
\begin{align}
  \nonumber |\nabla G(X)|
  &\overset{\eqref{identity:gradient}}{=} \left| \iint \nabla_X \Lambda(X,Y) G_0(Y) \, dY \right| \\[4pt]
  \nonumber &\overset{\eqref{identity:gradient_zero}}{=} \left| \iint \nabla_X \Lambda(X,Y) \big(G_0(Y) -[G_0]_{B_X}\big)\, dY \right|\\[4pt]
  \nonumber &\overset{\eqref{eq3.11a}}{\lesssim} \delta(X)^{-n-2} \iint_{B_X} \big|G_0(Y) -[G_0]_{B_X}\big| \, dY \\
  \label{eq3.11} &\lesssim \delta(X)^{-n-1} \iint_{B_X} |\nabla G_0(Y)|\,dY \,,
 \end{align}
where we have used also \eqref{eq3.4}, and Poincar\'e's inequality for BV (see \cite[Theorem 1, p.\ 189]{evansgariepy}). 

Now let $\hat{x}\in \pom$ be a ``touching point" for $X$, i.e.\ $|X-\hat{x}|=\delta(X)$. Then
\begin{equation*}
  \delta(X)^{-n-1}  \iint_{B_X} |\nabla G_0(Y)| \, dY
  \lesssim \delta(X)^{-n-1} \iint_{B\left(\hat{x}, 2\delta(X)\right)\cap\Omega} |\nabla G_0(Y)| \, dY
  \lesssim  \frac{C_\mu}{\delta(X)}\,,
\end{equation*}
  by hypothesis.
Combining the latter estimate with \eqref{eq3.11}, we obtain the desired conclusion.  We remark that the full strength of 
the Carleson measure condition was not required here, but only the weaker estimate
\begin{align*}
  \delta(X)^{-n} \iint_{B_X} |\nabla G_0(Y)|\,dY\leq C\,.
\end{align*}
\end{proof}

\begin{lemma}
  \label{lemma:carleson_regularization}
  If $G_0 \in BV_{\text{loc}}(\Omega)$ and $\mu = |\nabla G_0(Y)| \, dY$ is a Carleson measure, then also $|\nabla G(Y)| \, dY$ is a Carleson measure and
  \begin{align*}
    \sup_{r > 0, z \in \partial \Omega} \frac{1}{r^n} \iint_{B(z,r) \cap \Omega} |\nabla G(X)| \, dX \lesssim C_\mu
  \end{align*}
  where $C_\mu$ is the constant in \eqref{constant:carleson_measure_constant}.
\end{lemma}

\begin{proof}
  Fix $B(z,r)$ with $z\in \pom$. We cover $B(z,r) \cap \Omega$ by (possibly disconnected) ``half-open" regions 
  \begin{align*}
    V_k \coloneqq \left\{X\in \Omega \cap B(z,r): 2^{-k-1}r  \leq \delta(X) < 2^{-k}r\right\}\,,
  \end{align*}
  so that $\Omega \cap B(z,r) = \cup_{k=0}^\infty V_k$.  Observe that for $X\in V_k$, the ball $B_X$ defined in \eqref{eq3.4} is contained in
  \begin{align*}
    V^*_k \coloneqq \left\{Y\in \Omega \cap B(z,2r): 2^{-k-2}r  \leq \delta(Y) < 2^{-k+1}r\right\}\,,
  \end{align*}
  and moreover, that for $Y\in B_X$, we have
  \begin{align*}
    |X-Y|\leq \delta(X)/2 \approx \delta(Y)\,.
  \end{align*}
  Thus, using  \eqref{eq3.11}, we see that
  \begin{align*}
    \iint_{V_k} |\nabla G(X)| \, dX
    &\lesssim  \iint_{V_k} \delta(X)^{-n-1}  \iint_{B_X} |\nabla G_0(Y)| \, dY \, dX \\
    &\lesssim  \iint_{V^*_k}  |\nabla G_0(Y) |\left( \delta(Y)^{-n-1}  \iint_{|Y-X| \lesssim \delta(Y)} \, dX \right) \, dY
    \approx  \iint_{V^*_k}  |\nabla G_0(Y)| \, dY.
  \end{align*}
  Summing in $k$, and using that the sets $V_k^*$ have bounded overlaps, we obtain  
  \begin{align*}
    \iint_{B(z,r) \cap \Omega} |\nabla G(X)| \, dX
    \lesssim \iint_{B(z,2r) \cap \Omega} |\nabla G_0(Y)| \, dY
    \lesssim C_\mu r^n\,,
    \end{align*}
    as desired.
\end{proof}

\begin{lemma}
  \label{lemma:smooth_nt_convergence}
  If $G_0$ converges to $g(x)$ non-tangentially in the standard sense (respectively, in the one-sided sense) in a cone with large enough aperture, then also $G$ converges to $g(x)$ non-tangentially in the standard (respectively, one-sided) sense.
\end{lemma}

\begin{proof}
  Suppose that $Y \in \widetilde{\Gamma}^m(x)$ for some $m > 1$. We recall that
  \begin{align*}
    G(Y) = \frac{1}{\beta(Y)^{n+1}} \iint_{B(Y,\frac{\beta(Y)}{2m_2})} \zeta\left( \frac{Y-Z}{\beta(Y)} \right) G_0(Z) \, dZ.
  \end{align*}
  In particular, since $\dist(x,Y) < m \delta(Y)$, we have
  \begin{align*}
    \dist(x,Z) \le \dist(x,Y) + \dist(Y,Z) < m \delta(Y) + \frac{\beta(Y)}{2m_2} \le \left( m + \frac{1}{2} \right) \delta(Y)
  \end{align*}
  for every $Z \in B(Y,\beta(Y)/2m_2)$. Also, if $|G_0(Z) - g(x)| < \eps$ for every $Z \in B(Y,\tfrac{\beta(Y)}{2m_2})$, we can use the facts that $\iint \zeta = 1$ and $\zeta(X) \le 1$ for every $X \in \Omega$ to show that
  \begin{align*}
    |G(Y) - g(x)|
    \le \frac{1}{\beta(Y)^{n+1}} \iint_{B(Y,\frac{\beta(Y)}{2m_2})} \left| \zeta\left( \frac{Y-Z}{\beta(Y)} \right) \right| |G_0(Z) - g(x)| \, dZ
    \lesssim \eps.
  \end{align*}
  By combining these two observations we see that if $G_0$ converges to $g(x)$ non-tangentially in a cone with aperture $m$, then $G$ converges to $g(x)$ non-tangentially in a cone with aperture $m - \tfrac{1}{2}$. Observe that the preceding argument applies in 
  the case of either standard or one-sided non-tangential convergence.
\end{proof}

\begin{remark}
  The aperture of the cones does not play an important role in this paper and we use Lemma \ref{lemma:smooth_nt_convergence} without considering details related to them in the proofs. This is because we can always use mollifiers that are supported on a smaller ball than $B(0,\tfrac{1}{2m_2})$ and we use dyadic cones that we can construct in such a way that they contain cones of the type $\widetilde{\Gamma}^m$ for a large $m$ (see Section \ref{section:carleson_boxes}). 
\end{remark}

\section{Bilateral corona decomposition and one-sided non-tangential traces}
\label{section:bilateral_corona}

In $\R^{n+1}_+$, the construction of dyadic Carleson boxes and dyadic Whitney regions is very simple: just take a dyadic cube on $\R^n$, build a cube on top of it to get the Carleson box and remove the lower half of the cube to get the Whitney region. These objects are easy to work with particularly due to their simple geometric structure and they are very effective in many situations (see e.g. \cite{hofmannkenigmayborodapipher, hytonenrosen}). 
However, it is still possible to construct substitutes for these boxes and regions that share many good properties with their $\R^{n+1}_+$-analogues \cite[Section 3]{hofmannmartellmayboroda}.

In this paper, we need two versions of the Whitney regions from \cite{hofmannmartellmayboroda} for two different purposes:
\begin{enumerate}
  \item[1)] the original regions in a slightly modified form to prove Theorem \ref{theorem:bounded_extension} in Section \ref{section:proof_of_theorem_bounded},
  
  \item[2)] simplified and non-dilated regions for the construction of the extension of Proposition \ref{proposition:dyadic_extension}.
\end{enumerate}
The reason why we need these simplified regions is that although the boundaries of the original dilated regions are ADR, they are not quite neat enough for some more delicate estimates. We construct these regions in Sections \ref{section:carleson_boxes} and \ref{section:carleson_regions}.

Let us start by recalling some key tools from \cite{hofmannmartellmayboroda}. In this section, $\Omega \subset \R^{n+1}$ is an open set with $n$-UR boundary $\pom$ and $\D$ is a dyadic system on $\pom$. 
We begin with a standard 
Whitney decomposition of $\Omega$.  

\subsection{Whitney cubes and regions}
\label{section:whitney_cubes}

We use Whitney cubes and Whitney regions in our proofs and constructions throughout the article. Suppose that $\Wc \coloneqq \{I\}_I$ is a Whitney decomposition of $\Omega$ (see e.g. \cite[Chapter VI]{stein}, that is, $\{I\}_I$ is a collection of closed $(n+1)$-dimensional Euclidean cubes whose interiors are disjoint such that $\bigcup_I I = \Omega$ and
\begin{align*}
  4 \diam(I) \le \dist(4I,\pom) \le \dist(I,\pom) \le 40\diam(I) \, \text{ for every } I \in \Wc
\end{align*}
and 
\begin{align*}
  \frac{1}{4} \diam(I_1) \le \diam(I_2) \le 4 \diam(I_1)
\end{align*}
whenever $I_1 \cap I_2 \neq \emptyset$. For parameters $\eta$ and $K$ satisfying $\eta \ll 1 \ll K$ and for every $Q \in \D(\pom)$ we set
\begin{align}
  \label{defin:whitney_choice}
  \Wc^0_Q \coloneqq \Wc^0_Q(\eta, K) \coloneqq \{I \in \Wc \colon \eta^{1/4} \ell(Q) \le \ell(I) \le K^{1/2} \ell(Q), \dist(I,Q) \le K^{1/2} \ell(Q)\}.
\end{align}

\begin{remark}\label{remark:E-cks} 
  We note that $\Wc^0_Q$ is non-empty, for $\eta$ chosen small enough, and $K$ large enough, provided that $\Omega$ satisfies the corkscrew condition (see \cite[Section 3]{hofmannmartellmayboroda}). In particular, the latter is true when $\Omega=\Omega_E \coloneqq \ree\setminus E$, where $E\subset \ree$ is an n-ADR set. In the sequel, we shall always assume that $\eta$ and $K$ have been so chosen.
\end{remark}

\begin{defin}\label{I-dilate} 
  For $\xi > 1$ and every $I \in \Wc$, we let $I^*$ be the concentric dilation of $I$:
  \begin{align*}
    I^* = I^*(\xi) \coloneqq \xi I.
  \end{align*}
  We note that if $\xi$ is close enough to $1$, (and we shall always choose it so), the fattened cubes $I^*$ have bounded overlaps, and retain the property that $\diam(I^*) \approx \dist(I^*,\pom)$. We shall refer to such values of $\xi$ as \emph{allowable}.
\end{defin}

If we choose (as above) the parameters $\eta$, $K$ and $\xi$ in a suitable way, the 
collections $\bigcup_{I \in \Wc_Q} I$ and $\bigcup_{I \in \Wc_Q} I^*$, and certain variants of these collections, have strong geometric properties that we will formulate in the next lemmas and use in the subsequent sections.

\begin{defin}
  We say that a subcollection $\Sc \subset \D$ is \emph{coherent} if the following three conditions hold.
  \begin{enumerate}
    \item[(a)] There exists a maximal element $Q(\Sc) \in \Sc$ such that $Q \subset \Sc$ for every $Q \in \Sc$.
    \item[(b)] If $Q \in \Sc$ and $P \in \D$ is a cube such that $Q \subset P \subset Q(\Sc)$, then also $P \in \Sc$.
    \item[(c)] If $Q \in \Sc$, then either all children of $Q$ belong to $\Sc$ or none of them do.
  \end{enumerate}
  If $\Sc$ satisfies only conditions (a) and (b), then we say that $\Sc$ is \emph{semicoherent}.
\end{defin}

\begin{lemma}[{\cite[Lemma 2.2]{hofmannmartellmayboroda}}]
  \label{lemma:bilateral_corona}
  For any pair of positive constants $\eta \ll 1$ and $K \gg 1$ there exists a disjoint decomposition $\D = \Gc \cup \Bc$ satisfying the following properties:
  \begin{enumerate}
    \item[(1)] The ``good'' collection $\Gc$ is a disjoint union of coherent stopping time regimes $\Sc$.
    \item[(2)] The ``bad'' collection $\Bc$ and the maximal cubes $Q(\Sc)$ satisfy a Carleson packing condition:
               for every $Q \in \D$ we have
               \begin{align*}
                 \sum_{Q' \subset Q, Q' \in \Bc} \sigma(Q') + \sum_{\Sc: Q(\Sc) \subset Q} \sigma(Q(\Sc)) \le C_{\eta, K} \sigma(Q).
               \end{align*}
    \item[(3)] For every $\Sc$, there exists an $n$-dimensional Lipschitz graph $\Gamma_\Sc$, with Lipschitz constant at most $\eta$, such that
               for every $Q \in \Sc$ we have
               \begin{align*}
                 \sup_{x \in \Delta_Q^*} \dist(x,\Gamma_\Sc) + \sup_{y \in B_Q^* \cap \Gamma_\Sc} \dist(y,\pom) < \eta \ell(Q),
               \end{align*}
               where $B_Q^* \coloneqq B(x_Q,K\ell(Q))$ and $\Delta_Q^* \coloneqq B_Q^* \cap \pom$.
  \end{enumerate}
\end{lemma}
We call the decomposition $\D = \Gc \cup \Bc$ in Lemma \ref{lemma:bilateral_corona} the \emph{bilateral corona decomposition} of $\D$.

Next, we recall a construction in  \cite[Section 3]{hofmannmartellmayboroda}, leading up to and including in particular
\cite[Lemma 3.24]{hofmannmartellmayboroda}.   We summarize this construction as follows. 
\begin{lemma}
  \label{lemma:nta_domains}
  Let $E\subset \ree$ be UR, and set $\Omega_E \coloneqq \ree\setminus E$.  Given positive constants $\eta\ll 1$ and $K \gg 1$, as in \eqref{defin:whitney_choice} and Remark \ref{remark:E-cks}, let $\dd = \G\cup\B$, be the corresponding bilateral Corona decomposition of Lemma \ref{lemma:bilateral_corona}. Then for each $\Sc\subset \G$, and for each $Q\in \Sc$, the collection $\Wc^0_Q$ in \eqref{defin:whitney_choice} has an augmentation $\Wc^*_Q\subset \Wc$ satisfying the following properties.
  \begin{enumerate}
    \item $\Wc^0_Q\subset \Wc^*_Q = \Wc_Q^{*,+}\cup\Wc_Q^{*,-}$, where (after a suitable rotation of coordinates) each $I \in \Wc_Q^{*,+}$ lies above the Lipschitz graph $\Gamma_{\Sc}$ of Lemma \ref{lemma:bilateral_corona},  each $I \in \Wc_Q^{*,-}$ lies below $\Gamma_{\Sc}$. Moreover, if $Q'$ is a child of $Q$, also belonging to $\Sc$, then each $I \in \Wc_Q^{*,+}$ (resp.\ $I \in \Wc_Q^{*,-}$) belongs to the same connected component of $\om_E$ as each $I' \in \Wc_{Q'}^{*,+}$ (resp.\ $I' \in \Wc_{Q'}^{*,-}$) and  $\Wc_{Q'}^{*,+}\cap \Wc_{Q}^{*,+}\neq \emptyset$ (resp. $\Wc_{Q'}^{*,-}\cap\Wc_{Q}^{*,-}\neq \emptyset$).

    \item There are uniform constants $c$ and $C$ such that
    \begin{equation}\label{eq2.whitney2}
      \begin{array}{c}
        c\eta^{1/2} \ell(Q)\leq \ell(I) \leq CK^{1/2}\ell(Q)\,, \quad \forall I\in \mathcal{W}^*_Q, \\[5pt]
        \dist(I,Q)\leq CK^{1/2} \ell(Q)\,,\quad\forall I\in \mathcal{W}^*_Q, \\[5pt]
        c\eta^{1/2} \ell(Q)\leq\dist(I^*(\tau),\Gamma_{\Sc})\,,\quad \forall I\in \mathcal{W}^*_Q\,,\quad \forall \tau\in (0,\tau_0]\,.
      \end{array}
    \end{equation}

    \item For $\xi>1$, and recalling Definition \ref{I-dilate}, set 
    \begin{equation}\label{eq3.3aa}
      \Uc^\pm_Q=\Uc^\pm_{Q,\xi} \coloneqq \bigcup_{I\in \W^{*,\pm}_Q} 
      I^*(\xi)\,,\qquad \Uc_Q \coloneqq \Uc_Q^+\cup \Uc_Q^-\,,
    \end{equation}
    and given $\Sc'$, a non-empty semi-coherent subregime of $\Sc$, define 
    \begin{equation}\label{eq3.2}
      \Omega_{\Sc'} \coloneqq \Omega_{\Sc'}^+\cup \Omega_{\Sc'}^-\,,\quad
      \Omega_{\Sc'}^\pm = \Omega_{\Sc'}^\pm(\xi) \coloneqq {\rm int }\bigcup_{Q\in\Sc'} \Uc_Q^{\pm}\,.
    \end{equation}
    Then there exists $\xi_0 > 1$ such that each of $\Omega^\pm_{\Sc'}$ is a CAD (Definition \ref{CAD}), with chord-arc constants depending only on $n,\xi,\eta, K$, and the ADR/UR constants for $E$, provided that $1 < \xi < \xi_0$.
  \end{enumerate}
\end{lemma}
As in \cite{hofmannmartellmayboroda}, it will be useful for us to extend the definition of the Whitney region $\Uc_Q$ to the case that $Q\in\B$, the ``bad'' collection of Lemma \ref{lemma:bilateral_corona}. Let $\W_Q^*$ be the augmentation of $\W_Q^0$ as constructed in Lemma \ref{lemma:nta_domains}, and set
\begin{equation}\label{Wdef}
  \W_Q \coloneqq \left\{
    \begin{array}{l}
      \W_Q^*\,, \,\,Q\in\G, \\[6pt]
      \W_Q^0\,, \,\,Q\in\B
    \end{array}
  \right.\,.
\end{equation}
For $Q \in\G$ we shall henceforth simply write $\Wc_Q, \W_Q^\pm$ in place of $\Wc_Q^*,\W_Q^{*,\pm}$. For arbitrary $Q\in \dd$, good or bad, we may then make the following definitions.
\begin{defin}\label{def4.1}
  Given $\xi'>\xi>1$, we let $I^* = \xi I$ and $I^*_{\fat} = \xi' I$ denote dilated Whitney cubes, for allowable values of $\xi',\xi$ as in Definition \ref{I-dilate}.
  Suppose that $x \in \partial \Omega$ and $Q \in \D$. The \emph{closed Whitney region relative to $Q$}, and its fattened version are, respectively, the sets
  \begin{align*}
    \Uc_Q \coloneqq \bigcup_{I \in \Wc_Q} I^*,\qquad   \Uc^{\fat}_Q \coloneqq \bigcup_{I \in \Wc_Q} I^*_{\fat}\,.
  \end{align*}
  Similarly, we define standard and fattened versions of the \emph{``semi-closed" (i.e., closed away from $\pom$) truncated dyadic cone at $x$}:
  \begin{align*}
    \Upsilon_Q(x) \coloneqq \bigcup_{Q' \in \D_Q, x \in Q'} \Uc_{Q'},\qquad   
    \Upsilon^{\fat}_Q(x) \coloneqq \bigcup_{Q' \in \D_Q, x \in Q'} \Uc^{\fat}_{Q'}
  \end{align*}  
  and the \emph{``semi-closed" Carleson box relative to $Q$}:
  \begin{align*}
    \Tc_Q \coloneqq \bigcup_{Q' \in \D, Q' \subseteq Q} \Uc_{Q'}\,, \qquad
       \Tc^{\fat}_Q \coloneqq \bigcup_{Q' \in \D, Q' \subseteq Q} \Uc^{\fat}_{Q'}\,.
  \end{align*}
\end{defin}
We list some further properties of $\Uc_Q$ and $\Tc_Q$ in the next lemma. Most properties in the first lemma follow directly from the construction but some of them require slightly trickier estimates related to the choice of $\eta$ and $K$ and the bilateral corona decomposition (see \cite[Section 3]{hofmannmartellmayboroda}).

For an open set $\Omega \subset \ree$ that satisfies an interior corkscrew condition and has $n$-dimensional UR boundary $\pom$, we define the Whitney regions $\Uc_Q$ as above, but only 
include only those connected components contained in $\Omega$ (by the corkscrew condition, there must be at least one such).  Of course, this includes the case that $\Omega =\Omega_E =\ree\setminus E$, with for an $n$-dimensional UR set $E=\pom_E$, as in Lemma \ref{lemma:nta_domains}. 

\begin{lemma} 
  \label{lemma:properties_of_whitney_regions} Let $\Omega\subset \ree$ satisfy an interior corkscrew condition, 
  with $n$-dimensional UR boundary $\pom$. We have the following properties:
  \begin{enumerate}
    \item[$\bullet$] The region $\Uc_Q$ is a union of a uniformly bounded number of Whitney cubes $I$ such that $\ell(Q) \approx \ell(I)$ and $\dist(Q,I) \approx \ell(Q)$.
    
    \item[$\bullet$] The regions $\Uc_Q$ have a bounded overlap property, i.e. we have $\sum_i |\Uc_{Q_i}| \approx |\bigcup_i \Uc_{Q_i}|$ for cubes $Q_i$ such that $Q_i \neq Q_j$ if $i \neq j$.
    
    \item[$\bullet$] If $\Uc_Q \cap \Uc_P \neq \emptyset$, then $\ell(Q) \approx \ell(P)$ and $\dist(Q,P) \lesssim \ell(Q)$.
    
    \item[$\bullet$] For every $Y \in \Uc_Q$ we have $\delta(Y) \approx \ell(Q)$.
    
    \item[$\bullet$] For every $Q \in \D$, we have $|\Uc_Q| \approx \ell(Q)^{n+1} \approx \ell(Q) \cdot \sigma(Q)$.
    
    \item[$\bullet$] If $\diam(\pom) \approx \diam(\Omega)$, then $\Omega = \bigcup_{Q \in \D} \Tc_Q$.
    
    \item[$\bullet$] If $\diam(\pom) < \infty$ and $\diam(\Omega) = \infty$, then there exist $R \gtrsim \diam(\pom)$ and a ball $B(x,R)$ for some 
                     $x \in \pom$ such that $\pom \subset B(x,R)$ and $B(x,R) \setminus \pom \subset \bigcup_{Q \in \D} \Tc_Q$.
    
    \item[$\bullet$] If $Q \in \Gc$, then $\Uc_Q$ has at least one 
    connected component, and at most two, corresponding to $\Uc_Q^\pm$ in Lemma \ref{lemma:nta_domains}.
    
    \item[$\bullet$] If $Q \in \Bc$, then $\Uc_Q$ has a uniformly bounded number of connected components.
  \end{enumerate}
\end{lemma}

\subsection{Non-tangential convergence of $\eps$-approximators}
We shall use the properties in Lemma \ref{lemma:properties_of_whitney_regions} to prove some results about non-tangential convergence of $\eps$-approximators.

\begin{lemma}
  \label{lemma:two-sided_cones} Let $\Omega \subset \ree$ be as in Lemma \ref{lemma:properties_of_whitney_regions}, and write $\dd=\Bc\cup\Gc$ as in Lemmas \ref{lemma:bilateral_corona} and \ref {lemma:nta_domains}. Let $Q_0 \in \D$ be a fixed cube, denote
  \begin{align*}
    \Ms_{Q_0} \coloneqq \{Q \in \Bc \colon Q \subseteq Q_0\} \cup \{Q(\Sc) \colon Q(\Sc) \subseteq Q_0\}_{\Sc\in\Gc}
  \end{align*}
  and set
  \begin{align*}
    G_{Q_0}(x) \coloneqq \sum_{Q \in \Ms_{Q_0}} 1_Q(x)
  \end{align*}
  for every $x \in \pom$ (thus $G_{Q_0}$ vanishes outside of $Q_0$). Then $G_{Q_0}(x) < \infty$ for almost every $x \in Q_0$. In particular, for almost every $x \in Q_0$, there exists a stopping time regime $\Sc_x$ such that if $x \in Q$ and $\ell(Q) \le \ell(Q(\Sc_x))$, then $Q \in \Sc_x$. For each $Q \in \Sc_x$, the interior of the cone $\Upsilon_Q(x)$ splits into at most two chord-arc domains, as does the sawtooth region $\Omega_{\Sc_x}$.
\end{lemma}

\begin{proof}
  Since the collection $\Ms_{Q_0}$ satisfies a Carleson packing condition by Lemma \ref{lemma:bilateral_corona} and $Q \subset Q_0$ for every $Q \in \Ms_{Q_0}$, we have
  \begin{align*}
    \int_{Q_0} G_{Q_0}(x) \, d\sigma(x) = \sum_{Q \in \Ms_{Q_0}} \sigma(Q) \lesssim \sigma(Q_0).
  \end{align*}
  In particular, $G_{Q_0}(x) < \infty$ for almost every $x \in Q_0$. Thus, for almost every 
  $x \in Q_0$ there exist $C_x > 0$ such that if $x \in Q$ and $\ell(Q) < C_x$, then $Q \notin \Ms_{Q_0}$. In particular, there exists a stopping time regime $\Sc_x$ given by Lemma \ref{lemma:bilateral_corona} such that if $x \in Q$ and $\ell(Q) < C_x$, then $Q \in \Sc_x \subset \Gc$. Thus, by Lemma \ref{lemma:properties_of_whitney_regions}, the corresponding Whitney region $\Uc_Q$ splits into at most two connected components. The final property follows now from Lemma \ref{lemma:nta_domains}.
\end{proof}

For every $x \in \pom$ that satisfies the condition in Lemma \ref{lemma:two-sided_cones}, we denote the components of $\Upsilon_Q(x)$ by $\Upsilon_Q^\pm(x)$, whose interiors, denoted by $\widetilde{\Upsilon}_Q^\pm(x)$, are subdomains of $\Omega_{\Sc_x}^\pm$ (see \eqref{eq3.2}), respectively. Since $\Omega$ satisfies the corkscrew condition, at least one of $\Omega_{\Sc_x}^\pm$ is contained in $\Omega$, and it may be that both are. We define $\Upsilon^{+,\fat}_{Q(\Sc_x)}$ in the same way.

\begin{lemma}
  \label{lemma:two-sided_traces} Let $\Omega\subset \ree$ be an open set satisfying an interior corkscrew 
  condition and let $\pom$ be UR. Suppose that $\Phi \colon \Omega \to \R$ is a smooth function such that $\mu = |\nabla \Phi(Y)| \, dY$ is a Carleson measure, and $|\nabla \Phi(X)| \lesssim \tfrac1{\delta(X)}$ for every $X \in \Omega$. Then $\Phi$ has one-sided non-tangential 
  boundary traces in the following sense: for $\sigma$-a.e.\ $x \in \pom$, the limits
  \begin{align*}
    \varphi^+(x) \coloneqq \lim_{Y \in \widetilde{\Upsilon}_{Q(\Sc_x)}^+(x), Y \to x} \Phi(Y) \ \ \ \text{ and } \ \ \ \varphi^-(x) \coloneqq \lim_{Y \in \widetilde{\Upsilon}_{Q(\Sc_x)}^-(x), Y \to x} \Phi(Y)
  \end{align*}
  exist and satisfy $\|\varphi^\pm\|_{L^\infty(\pom)} \le \|\Phi\|_{L^\infty(\Omega)}$, provided that $\Omega_{\Sc_x}^\pm\subset \Omega$.  
\end{lemma}

\begin{remark}\label{r4.15}
  As noted above,  necessarily $\Omega_{\Sc_x}^\pm\subset \Omega$ for at least one choice of $+$ or $-$, and possibly both.  Thus, $\Phi$ has at least a 1-sided non-tangential trace a.e.\ on $\pom$.  In the case that both components of $\Omega_{\Sc_x}$ are contained in 
  $\Omega$, the traces $\varphi^+$ and $\varphi^-$ may not coincide. Indeed, if $\Omega = \R^{n+1}_+ \cup \R^{n+1}_-$ and $\Phi = 1_{\R^{n+1}_+} - 1_{\R^{n+1}_-}$, then $\varphi^+(x) = 1$ and $\varphi^-(x) = -1$ for every $x \in \pom$ (when we have chosen the directions $+$ and $-$ in the obvious way).
\end{remark}

\begin{proof}[Proof of Lemma \ref{lemma:two-sided_traces}] Fix a cube $Q_0\in \dd$, and
  let $x \in \pom$ be a point satisfying the condition $G_{Q_0}(x)<\infty$ in Lemma \ref{lemma:two-sided_cones}. 
  We suppose that $\Omega_{\Sc_x}^+ \subset \Omega$, and
  consider the limit in $\widetilde{\Upsilon}_{Q(\Sc_x)}^+(x)$;  the case that
  $\Omega_{\Sc_x}^- \subset \Omega$ may be handled by the same argument.
   Let $\{X_k\}_k$ be an arbitrary sequence of points in $\widetilde{\Upsilon}_{Q(\Sc_x)}^+(x)$ such that $X_k \to x$.
 It suffices to show that $\{\Phi(X_k)\}$ is a Cauchy sequence. 

  We have fixed $1<\xi<\xi'$, and have constructed the corresponding standard and ``fat'' versions of the Whitney regions, cones and Carleson boxes as in Definition \ref{def4.1}.
  Using Lemma \ref{lemma:nta_domains}, we set
  \begin{align*}
    \Upsilon_0 \coloneqq \Upsilon^{+,\fat}_{Q(\Sc_x)} \,.
  \end{align*}
  Thus, $\Upsilon^+_{Q(\Sc_x)} \subset \Upsilon_0$, and the interior of $\Upsilon_0$ is an NTA domain. Let $k, m \in \N$, $m \ge k$, and let $0<\eps \ll \xi'-\xi$.
  Since $X_k, X_m \in \Upsilon^+_{Q(\Sc_x)}$, there exists a chain of balls $\{B_i\}_{i=1}^N$, $B_i \coloneqq B(Y_i,r_i)$, inside the interior of
  $\Upsilon_0$, with the following properties:
  \begin{enumerate}
    \item[(i)] $Y_1 = X_k$, $Y_N = X_m$,
    
    \item[(ii)] $r_1 \le \eps \delta(X_k)$, $r_N \le \eps \delta(X_m)$,

    \item[(iii)] $B_i \cap B_{i+1} \neq \emptyset$ for every $i\geq 1$,

    \item[(iv)] $r_i \approx \delta(Y_i) \approx \dist(B_i,\pom)$,
    
      \item[(v)] $1/4 \leq r_i/r_{i+1} \leq 4$, 

    \item[(vi)] for each $i\geq 1$, $B_i \cup B_{i+1} \subset \Cmf_i \subset \Upsilon_0$, where $\Cmf_i$ is a 
    cylinder with height $h_i$ and radius $\rho_i$ satisfying
    \[ 1\leq h_i/r_i \leq 8\,,\quad 1\leq \rho_i/r_i \leq 8\,,\]
     and such that
     $\dist(\Cmf_i,\pom)\approx \diam(\Cmf_i) \approx r_i$,

    \item[(vii)] the balls $\{B_i\}_i$ and the cylinders $\{\Cmf_i\}_i$ have bounded overlaps.
  \end{enumerate}
  Here, the implicit constants depend on the NTA properties of $\Upsilon_0$, and possibly on $\eps$.
  
  We now have
  \begin{align*}
    |\Phi(X_k) - \Phi(X_m)| \le &\fiint_{B_1} |\Phi(X) - \Phi(X_k)| \, dX
    + \sum_{i=1}^{N-1} \left| \fiint_{B_i} \Phi(X) \, dX - \fiint_{B_{i+1}} \Phi(X) \, dX \right| \\
    &\ \ \ \ + \fiint_{B_N} |\Phi(X) - \Phi(X_m)| \, dX \\
    \coloneqq &I_1 + I_2 + I_3.
  \end{align*}
  By the mean value theorem and the pointwise gradient bound, we know that $\Phi$ is locally Lipschitz. Thus,
  \begin{align*}
    I_1 \leq C \fiint_{B_1} \frac{|X - X_k|}{\delta(X_k)} \, dX
        \le  C \fiint_{B_1} \frac{\eps\delta(X_k)}{\delta(X_k)} \, dX
        = C\eps
  \end{align*}
  and similarly $I_3 \lesssim \eps$. As for $I_2$, by  (v) and (vi) above, and 
  Poincar\'e's inequality, we have
  \begin{align*}
    \left| \fiint_{B_i} \Phi(X) \, dX - \fiint_{B_{i+1}} \Phi(X) \, dX \right|
    &= \left| \fiint_{B_i} \Phi(X) \, dX - \langle \Phi \rangle_{\Cmf_i} + \langle \Phi \rangle_{\Cmf_i} - \fiint_{B_{i+1}} \Phi(X) \, dX \right| \\
    &\le 2 \fiint_{\Cmf_i} \left| \Phi(X) - \fiint_{\Cmf_i} \Phi(Y) \, dY \right| \, dX \\
    &\lesssim \frac{r_i}{|\Cmf_i|} \iint_{\Cmf_i} |\nabla \Phi(X)| \, dX \\
    &\lesssim \frac{1}{\delta(Y_i)^n} \iint_{\Cmf_i} |\nabla \Phi(X)| \, dX \\
    &\lesssim \iint_{\Cmf_i} |\nabla \Phi(X)| \delta(X)^{-n} \, dX.
  \end{align*}
  
  By construction, we may choose $Q \in \Sc_x$, with $\ell(Q) \approx \max(\delta(X_k),\delta(X_m))$, such that $X_k, X_m \in \Upsilon_Q^+(x)$
  and  $\Cmf_i \subset  \Upsilon_Q^{+,\fat}(x)$ for each $i=1,2,\dots,N$. Then, by the bounded overlap property
  of the cylinders $\{\Cmf_i\}_i$ ,and the structure of the dyadic cones, we have
  \begin{align*}
    I_2
    \lesssim \sum_{i=1}^N \iint_{\Cmf_i} |\nabla \Phi(X)| \delta(X)^{-n} \, dX
    &\lesssim \iint_{\Upsilon_Q^{+,\fat}(x)} |\nabla \Phi(X)| \delta(X)^{-n} \, dX \\
    &\le \sum_{x \in Q' \in \D_Q} \iint_{\Uc_{Q'}^{+,\fat}} |\nabla \Phi(X)| \delta(X)^{-n} \, dX \\
    &\lesssim \sum_{x \in Q' \in \D_Q} \iint_{\Uc_{Q'}^{+,\fat}} |\nabla \Phi(X)| \ell(Q')^{-n} \, dX \\
    &\lesssim \sum_{Q' \in \D_Q} \frac{1_{Q'}(x)}{\sigma(Q')} \iint_{\Uc_{Q'}^{+,\fat}} |\nabla \Phi(X)| \, dX.
  \end{align*}
  We notice that
  \begin{align*}
    \int_Q \sum_{Q' \in \D_Q} \frac{1_{Q'}(y)}{\sigma(Q')} \iint_{\Uc_{Q'}^{+,\fat}} |\nabla \Phi(X)| \, dX \ d\sigma(y)
    &= \sum_{Q' \in \D_Q} \iint_{\Uc_{Q'}^{+,\fat}} |\nabla \Phi(X)| \, dX \\
    &\lesssim \iint_{\Tc^{\fat}_Q} |\nabla \Phi(X)| \, dX
    \lesssim C_\mu \sigma(Q),
  \end{align*}
  since $|\nabla \Phi(X)| \, dX$ is a Carleson measure. Thus, $\sum_{Q' \in \D_Q} \frac{1_{Q'}(x)}{\sigma(Q')} \iint_{\Uc^{+,\fat}_{Q'}} |\nabla \Phi(X)| \, dX 
  < \infty$ for $\sigma$-a.e.\ $x \in \pom$. In particular,
  \begin{align*}
    \lim_{\ell(Q) \to 0} \sum_{Q' \in \D_Q} \frac{1_{Q'}(x)}{\sigma(Q')} \iint_{\Uc^{+,\fat}_{Q'}} |\nabla \Phi(X)| \, dX = 0
  \end{align*}
  for $\sigma$-a.e.\ $x \in \pom$. It follows that $I_2 \leq \eps$ if $k,m$ are large enough, and consequently that $I_1 + I_2 + I_3 \lesssim \eps$. We therefore conclude that $\{\Phi(X_k)\}_k$ 
  is a Cauchy sequence.
\end{proof}

\begin{remark}
  \label{r4.16}
  As noted above (see Remark \ref{r4.15}), it is possible that non-tangential traces, whose existence is guaranteed by Lemma \ref{lemma:two-sided_traces}, may exist from two sides, and they may not coincide. It will therefore be convenient to fix a canonical, unambiguous choice of non-tangential approach. To this end, we proceed as follows. Recall the counting function
  $G_{Q}$ defined in Lemma \ref{lemma:two-sided_cones}. Set
  \begin{align*}
    A_{\NT} \coloneqq \{ x \in \pom: G_Q(x) <\infty\,,\, \forall Q\in \dd\}.
  \end{align*}
  Recall that for each cube $Q$, $G_Q(x)<\infty$ for $\sigma$-a.e.\ $x \in \pom$. Since $\dd$ is countable, we find that $\sigma(\pom\setminus A_{\NT}) = 0$. For each $x\in A_{\NT}$, there is a stopping time regime $\Sc_x$, as in Lemma \ref{lemma:two-sided_cones}, with maximal cube $Q(\Sc_x)$. We set $\dd_{\NT} \coloneqq \{Q(\Sc_x)\}_{x\in A_{\NT}}$, and observe that this collection is countable (thus, $\Sc_x = \Sc_y$ for many choices of distinct $x$ and $y$). We enumerate $\dd_{\NT} = \{Q_i\}_{i=1}^\infty$, and for each $Q_i \in \dd_{\NT}$, we let $\Sc_i$ be the stopping time regime with maximal cube $Q_i$. If $\Omega_{\Sc_x}^\pm$ is contained in $\Omega$, then for every $\Phi$ as in Lemma \ref{lemma:two-sided_traces}, the non-tangential traces $\varphi^\pm(x)$ are defined for $\sigma$-a.e.\ $x \in A_{\NT}$. In addition, for $x \in A_{\NT}$, there is an index $i$ with $\Sc_x = \Sc_i$, and since the corkscrew condition holds in $\Omega$, at least one of $\Omega_{\Sc_i}^\pm$ is contained in $\Omega$. If there is only one such, then the trace $\varphi(x)$ is defined unambiguously; on the other hand, if both are contained in $\Omega$, then we arbitrarily set $\varphi(x) = \varphi^+(x)$. {\em
  Note that we make this same choice for every $x \in A_{\NT}$ such that $\Sc_x = \Sc_i$, and 
  moreover, that this choice is specified in advance, and is independent of $\Phi$.}
\end{remark}

\section{Some results on boundary behavior of bounded harmonic functions}
\label{section:boundary}

In this section, we shall prove some useful facts about  boundary behavior of bounded harmonic functions.  We begin with some preliminary observations.

\begin{remark}
  \label{r5.1}
  In the sequel, given a function $v$ defined in an open set $\Omega$, we let $\tmf v$ denote the non-tangential trace of $v$ on $\pom$, i.e., for $x\in\pom$, set
  \begin{equation}
    \label{eqNT}
    \tmf v(x) \coloneqq \lim_{Y \to x \text{ N.T.}} v(Y)\,,
  \end{equation}
  provided that this non-tangential limit exists. Here, the notation $Y \to x \text{ N.T.}$ means that $Y\to x$, with $Y\in \widetilde{\Gamma}(x)$ (see Definition \ref{defin:cones}), or with $Y\in \Gamma(x)$ (see Definition \ref{defin:dyadic_regions_and_cones} below, and also Remark \ref{r7.4}). We recall that in an NTA domain $\Omega$, if $v$ is a bounded harmonic function, then $\tmf v(x)$
  exists for $\omega$-a.e.\ $x\in\pom$, by virtue of the Fatou Theorem of \cite[Theorem 6.4]{jerisonkenig}, where $\omega$ is harmonic measure for $\Omega$ with any fixed pole.  Recall also that if, in addition, the NTA domain has an ADR boundary (i.e., so that $\Omega$ is a CAD; see Definition \ref{CAD}), then in particular, by results obtained independently in \cite{davidjerison} and in \cite{semmes}, $\omega$ and $\sigma = \mathcal{H}^n\lfloor_{\pom}$ are mutually absolutely continuous, and thus for a bounded harmonic function $v$, one has that $\tmf v(x)$ exists for $\sigma$-a.e.\ $x\in \pom$. In particular, in this context, the Dirichlet problem is uniquely solvable in $\Omega$, with data in $L^p(\pom,\sigma)$ for $p<\infty$ sufficiently large (depending on dimension and the chord-arc constants of $\Omega$), with $L^p$ control of the non-tangential maximal function, and with non-tangential convergence of the solution to the data, $\sigma$-a.e.\ on $\pom$.  Therefore, in a bounded chord-arc domain $\Omega$, if $v$ is a bounded harmonic function with non-tangential trace $\tmf v$, we then have
  \begin{equation}\label{hm-rep}
    v(Y)= \int_{\pom} \tmf v\, d\omega^Y\,,\quad \forall Y \in \Omega\,.
  \end{equation}
\end{remark}

\begin{lemma}\label{lemma-additive}
  Let $\Omega$ be a bounded CAD. Let $\{u_k\}_{k=1}^\infty$ be a sequence of non-negative, bounded harmonic functions in $\Omega$, whose sum
  \begin{align*}
    u \coloneqq \sum_{k=1}^\infty u_k
  \end{align*}
  is also bounded in $\Omega$. Then the non-tangential trace operator $\tmf$ satisfies the countable additivity property
  \begin{align*}
  \tmf u(x) = \sum_{k=1}^\infty \tmf u_k(x)\,,\quad \sigma\text{-\rm a.e.\ }  x\in \pom\,.
  \end{align*}
\end{lemma}

\begin{proof}
  Set
  \begin{align*}
    f \coloneqq \tmf u\,,\quad f_k \coloneqq \tmf u_k\,,
  \end{align*}
  which, as noted above, exist $\sigma$-a.e.\ on $\pom$, and of course inherit non-negativity from $u$ and $u_k$. Since $\tmf$ is a linear operator, for each positive integer $N$, and at $\sigma$-a.e.\ point on $\pom$,
  \begin{align*}
    \sum_{k=1}^N f_k  = \sum_{k=1}^N \tmf u_k= \tmf\Big(\sum_{k=1}^N u_k\Big) \leq  \tmf\Big(\sum_{k=1}^\infty u_k\Big) = \tmf u\,,
  \end{align*}
  where in the inequality we have used that $u_k\geq 0$ for every $k$. Letting $N\to \infty$, we find that 
  \begin{align*}
    \widetilde{f} \coloneqq \sum_{k=1}^\infty f_k \in L^\infty(\pom,\sigma)\,.
  \end{align*}
  Our goal is then to show that $f = \widetilde{f}$ at $\sigma$-a.e.\ point on $\pom$. To this end, since $\Omega$ is a bounded CAD, we may apply \eqref{hm-rep} to obtain
  \begin{align*}
    \int_{\pom} f\, d\omega^Y = u(Y)
                              = \sum_{k=1}^\infty u_k(Y)
                              &= \sum_{k=1}^\infty  \int_{\pom} f_k\, d\omega^Y \\
                              &= \int_{\pom}  \sum_{k=1}^\infty  f_k\, d\omega^Y 
                              = \int_{\pom} \widetilde{f}\, d\omega^Y \eqqcolon \widetilde{u}(Y)\,,
  \end{align*}
  for each $Y\in \Omega$, where the interchange of summation and integration in the fourth equality may be justified by monotone convergence, since $f_k\geq 0$. Thus $ \widetilde{u} = u$ at every point in $\Omega$, hence, $\sigma$-a.e.\ on $\pom$, we have
  \begin{align*}
    0 = \tmf(\widetilde{u}-u) = \tmf \,\widetilde{u}-\tmf u = \widetilde{f} - f\,.
  \end{align*}
 \end{proof}
 
In the sequel,  given a set $A$, we denote the usual supremum norm of a function $g$ defined 
on $A$ by
\begin{align*}
  \|g\|_{\sup(A)} \coloneqq \sup_{X\in A} |g(X)|\,.
\end{align*}
Of course, for continuous $g$, one has $\|g\|_{\sup(A)} = \|g\|_{L^\infty(A)}$; in particular,
\begin{equation}\label{eq5.1a}
  \|u\|_{\sup(\Omega)} = \|u\|_{L^\infty(\Omega)}\,,\quad \text{for } u \text{ harmonic in } \Omega\,.
\end{equation}

Next, we recall that by \cite[Theorem 3.9.1]{helms}, if $g$ is a Borel measurable function that is everywhere bounded on $\pom$ (i.e., $\|g\|_{\sup(\pom)} <\infty$), then
\begin{align*}
  v(Y) \coloneqq \int_{\pom} g \, d\omega^Y,
\end{align*}
exists and is harmonic in $\Omega$, and satisfies $\|v\|_{\sup(\Omega)} \le \|g\|_{\sup(\pom)}$.

Our main result in this section is the following.

\begin{lemma}\label{lemma-trace}
  Let $\Omega\subset \ree$ be an open set, with $n$-UR boundary. Suppose that $g$ is an everywhere bounded Borel measurable function on $\pom$.
  Set
  \begin{align}
    \label{eq5.6a} v_g(Y) \coloneqq \int_{\pom} g \, d\omega^Y\,.
  \end{align}
  Then the non-tangential trace $\tmf v_g$ exists $\sigma$-a.e.\ on $\pom$, and
  \begin{equation}
    \label{eq5.6} \tmf v_g (x)= g(x)\,,\quad \sigma\text{-\rm a.e.\ }  x\in \pom\,.
  \end{equation}
\end{lemma}
We note that no continuity assumption is imposed on $g$; moreover, in the generality of Lemma \ref{lemma-trace}, harmonic measure need {\em not} be absolutely continuous with respect to surface measure on $\pom$.

\begin{proof}
  By Lemma \ref{lemma:nta_domains} and Remark \ref{r4.16}, there is a countable collection of bounded chord-arc domains $\{\Omega^i\}_{i=1}^\infty$, with $\Omega^i = \Omega^{\pm}_{\Sc_i}$ for some choice of $\pm$, such that $\Omega^i\subset \Omega$, and
  \begin{equation}
    \label{ae-cad} \sigma\left(\pom \setminus \left(\cup_i \,\pom^i\right)\right)= 0\,.
  \end{equation}
  In the case that each of $ \Omega^{\pm}_{\Sc_i}$ is contained in $\Omega$, then we may choose $\Omega^i$ to be either of these. Moreover, by the Fatou theorem of \cite{bortzhofmann}, $\tmf v_g$ exists at $\sigma$-a.e.\ point on $\pom$; more precisely, it exists at $\sigma$-a.e.\ point on $\pom \cap \pom^i$, for each $i$, as a one-sided non-tangential trace (i.e., with the limit taken through the non-tangential approach region within $\Omega^i$); one may then invoke \eqref{ae-cad} to cover $\pom$ up to a set of $\sigma$-measure zero. Thus, it is enough to verify that \eqref{eq5.6} holds for $\sigma$-a.e.\ $x\in \pom \cap \pom_1$, where $\Omega_1$ is any bounded chord-arc subdomain of $\Omega$, whose boundary meets $\pom$.  We therefore fix such a subdomain $\Omega_1$, and let $\tmf_1$ denote the non-tangential trace operator on $\pom_1$.  Let $g$ be an everywhere bounded Borel measurable function on $\pom$, and define $v_g$ as in \eqref{eq5.6a}, so that $v_g$ is a bounded harmonic function in $\Omega$.

  We note that if $x \in \pom\cap\pom_1$ is a point where $\tmf v_g(x)$ exists, then $\tmf_1 v_g(x)$ exists, and
  \begin{equation}
    \label{eq5.9} \tmf_1 v_g(x) = \tmf v_g(x)\,,
  \end{equation}
  since the non-tangential approach region in the subdomain $\Omega_1$ is contained in a non-tangential approach region for the ambient domain $\Omega$.  Observe also that
  \begin{align*}
    \tmf_1 v_g (X) = v_g(X)\,,\quad X\in \Omega\cap\pom_1\,,
  \end{align*}
  since $v_g$ is, of course, continuous in $\Omega$. Applying \eqref{hm-rep} in the bounded chord-arc domain $\Omega_1$, we therefore have
  \begin{align*}
    v_g(Y) = \int_{\pom\cap\pom_1} \tmf_1 v_g\, d\omega_1^Y \,+\, \int_{\Omega\cap\pom_1}  v_g\, d\omega_1^Y\,,\quad Y \in \Omega_1\,,
  \end{align*}
  where $\omega_1$ is harmonic measure for $\Omega_1$. We also define
  \begin{align*}
    \widetilde{v}_g(Y) \coloneqq \int_{\pom\cap\pom_1} g\, d\omega_1^Y \,+\, \int_{\Omega\cap\pom_1}  v_g\, d\omega_1^Y\,,\quad Y \in \Omega_1\,,
  \end{align*}
 Thus, by Remark \ref{r5.1}, $v_g$ is the unique solution to the Dirichlet problem in $\Omega_1$ with  
  boundary data $(\tmf_1 v_g) 1_{\pom\cap\pom_1} + v_g 1_{\Omega\cap\pom_1}$, and $\widetilde{v}_g$ is the 
  unique solution to the Dirichlet problem in $\Omega_1$ with  
  boundary data $g 1_{\pom\cap\pom_1} + v_g 1_{\Omega\cap\pom_1}$.  Moreover,
  each of these solutions
  converges non-tangentially in $\Omega_1$ to its corresponding boundary data. In particular,
  \begin{align}
    \label{identity:boundary_traces}
    \tmf_1 v_g = (\tmf_1 v_g) 1_{\pom\cap\pom_1} + v_g 1_{\Omega\cap\pom_1} \quad \text{ and } \quad
    \tmf_1 \widetilde{v}_g = g 1_{\pom\cap\pom_1} + v_g 1_{\Omega\cap\pom_1},
  \end{align}
  $\sigma_1$-a.e.\ on $\pom_1$, where $\sigma_1 \coloneqq \Hc^n \lfloor_{\pom_1}$ is the surface measure on $\pom_1$.
  
  We now claim that $v_g = \widetilde{v}_g$ in $\Omega_1$. Assuming the claim momentarily, we then have $\tmf_1 v_g = \tmf_1 \widetilde{v}_g$, and this gives us $\tmf_1 v_g(x) = g(x)$ for $\sigma_1$-a.e.\ $x \in \pom_1$ by \eqref{identity:boundary_traces}.
  In particular, we have $\tmf v_g = g$ for $\sigma$-a.e.\ point on $\pom\cap\pom_1$ by \eqref{eq5.9}, and hence that \eqref{eq5.6} holds, as desired.
  
  It therefore remains to verify that $v_g = \widetilde{v}_g$ in $\Omega_1$.  To this end, we note first that the claim holds immediately in the special case that $g$ is continuous on $\pom$, since in that case $\tmf v_g = g$ at every point on $\pom$ (indeed, every boundary point is regular in the sense of Wiener, by the ADR property (see e.g. \cite[Lemma 3.27]{hofmannlemartellnystrom} or \cite[Section 3]{zhao})). By definition of $v_g$ and $\widetilde{v}_g$, we may write
  \begin{align*}
    v_g(Y) =\int_{\pom\cap\pom_1} \tmf_1 \left(\int_{\pom} g \,d\omega^{(\cdot)}\right)\, d\omega_1^Y \,+\, \int_{\Omega\cap\pom_1}  \int_{\pom} g \,d\omega^{X}\, d\omega_1^Y(X)\,,
  \end{align*}
  and also
  \begin{align*}
    \widetilde{v}_g(Y) =\int_{\pom\cap\pom_1} g \, d\omega_1^Y \,+\, \int_{\Omega\cap\pom_1}  \int_{\pom} g \,d\omega^{X}\, d\omega_1^Y(X)\,.
  \end{align*}
  For each $Y\in \Omega_1$, define two non-negative set functions on the Borel subsets of $\pom$ as follows:
  \begin{align*}
    \mu^Y(A) \coloneqq \int_{\pom\cap\pom_1} \tmf_1 \left(\omega^{(\cdot)}(A)\right)\, d\omega_1^Y \,+\, \int_{\Omega\cap\pom_1}  \omega^{X}(A)\, d\omega_1^Y(X)\,,
  \end{align*}
  and 
  \begin{align*}
    \widetilde{\mu}^{\,Y}(A) \coloneqq \int_{\pom\cap\pom_1} 1_A\, d\omega_1^Y \,+\, \int_{\Omega\cap\pom_1}  \omega^{X}(A)\, d\omega_1^Y(X)\,.
  \end{align*}
  Note that $\mu^Y(A)\leq 1$ and $\widetilde{\mu}^{\,Y}(A)\leq 1$ for all Borel $A\subset \pom$, since $\omega$ and $\omega_1$ are probability measures.  Since $g$ is Borel measurable, it suffices to show that $\mu^Y$ and $\widetilde{\mu}^{\,Y}$ are Borel measures, with $\mu^Y=\widetilde{\mu}^{\,Y}$ , for each $Y\in \Omega_1$; indeed, in that case we would have
  \begin{equation}
    \label{eq5.10} v_g(Y) = \int_{\pom} g\, d\mu^Y = \int_{\pom} g\, d\widetilde{\mu}^{\,Y} = \widetilde{v}_g(Y)\,,
  \end{equation}
  as claimed. Moreover, we have already observed that \eqref{eq5.10} holds in the special case that $g$ is continuous on $\pom$, thus it suffices simply to show that $\mu^Y$ and $\widetilde{\mu}^{\,Y}$ are Borel measures, since equality then follows by equality on the continuous functions; in turn, it therefore suffices to show that  $\mu^Y$ and $\widetilde{\mu}^{\,Y}$ are countably additive on the Borel subsets of $\pom$, i.e., that
  \begin{equation}
    \label{eq5.11} \mu^Y\Big(\bigcup_{k=1}^\infty A_k\Big) = \sum_{k=1}^\infty \mu^Y(A_k)\,,
  \end{equation}
  and similarly for $\widetilde{\mu}^{\,Y}$, whenever $\{A_k\}_k$ is a countable family of disjoint Borel subsets of $\pom$. To this end, given such a collection $\{A_k\}_k$, set $A \coloneqq \cup_k A_k$, and define
  \begin{align*} 
    u(X) \coloneqq \omega^{X} (A)\,, \quad u_k(X) \coloneqq \omega^X(A_k)\,.
  \end{align*}
  Since harmonic measure is a probability measure, and in particular is countably additive, we then have
  \begin{equation}
    \label{eq5.12} 1\geq u(X)= \sum_{k=1}^\infty u_k(X) \,,\quad \forall X \in \Omega\,.
  \end{equation}
  Recall that harmonic measure and surface measure are mutually absolutely continuous on the boundary of a chord-arc domain. Consequently, by \eqref{eq5.12}, Lemma \ref{lemma-additive} (applied in the
  bounded chord-arc domain $\Omega_1$), and monotone convergence, we find that
  \begin{multline*}
    \mu^Y(A) = \int_{\pom\cap\pom_1} \tmf_1 u\, d\omega_1^Y \,+\, \int_{\Omega\cap\pom_1}  u\, d\omega_1^Y\\[4pt]
    = \sum_{k=1}^\infty\left(\int_{\pom\cap\pom_1} \tmf_1 u_k\, d\omega_1^Y \,+\, \int_{\Omega\cap\pom_1}  u_k\, d\omega_1^Y\right) = \sum_{k=1}^\infty \mu^Y(A_k)\,.
  \end{multline*}
  The argument to treat $\widetilde{\mu}^{\,Y}$ is similar but simpler, requiring only countable additivity of harmonic measure in lieu of Lemma \ref{lemma-additive}, and we omit the details.
\end{proof}

\section{Proof of Theorem \ref{theorem:bounded_extension}}
\label{section:proof_of_theorem_bounded}

We now move to the proof of Theorem \ref{theorem:bounded_extension}. Although we can still 
follow the original strategy of Varopoulos \cite{varopoulos2}, consisting of $\eps$-approximation and iteration, we have to be more careful with our construction. For example,
the $\eps$-approximators in our setting may not have pointwise non-tangential boundary traces but rather only one-sided traces in the sense of Lemma \ref{lemma:two-sided_traces} (see Remark \ref{r4.15}). We shall therefore rely on the construction of an unambiguously defined (at least 1-sided) non-tangential trace, as outlined in Remark \ref{r4.16}. In addition, absolute continuity of harmonic measure with respect to surface measure may fail in the present generality, but Lemma \ref{lemma-trace} will allow us to make harmonic extensions, and to relate the non-tangential traces of these extensions to the data, thus allowing us to follow the basic strategy of Varopoulos.

In this section, $\Omega \subset \R^{n+1}$ is an open set with $n$-UR boundary $\pom$.

Suppose that $f$ is a Borel measurable function on $\pom$, with $\|f\|_{L^\infty(\pom,\sigma)}<\infty$. We will now construct the extension $\Phi$ in Theorem \ref{theorem:bounded_extension}. 

Since $f \in L^\infty(\pom,\sigma)$, there is a set $Z\in \pom$, with $\sigma(Z)= 0$, such that 
\begin{align*}
  \|f\|_{\sup(\pom\setminus Z)}= \|f\|_{L^\infty(\pom,\sigma)}\,.
\end{align*}
Since $\sigma$ is a Borel regular measure, there is a Borel set $Z_0\supset Z$, with $\sigma(Z_0)=0$. Set
\begin{align*}
 f_0(x) \coloneqq \left\{ \begin{array}{cl}
                         f(x)&\text{, if } x\in \pom\setminus Z_0 \\[4pt]
                         0 &\text{, if } x\in Z_0\,.
                       \end{array} \right. 
\end{align*}
Note that $f_0 = f$ at $\sigma$-a.e.\ point on $\pom$. Moreover, $f_0$ is an everywhere bounded, Borel measurable function on $\pom$, so by \cite[Theorem 3.9.1]{helms}, we know that $u_0 \colon \Omega \to \R$, defined by
\begin{align*}
  u_0(X) \coloneqq \int_{\pom} f_0(y) \, d\omega^X(y),
\end{align*}
is a harmonic function in $\Omega$ satisfying
\begin{align*}
  \|u_0\|_{\sup(\Omega)} \le \|f_0\|_{\sup(\pom)} = \|f\|_{L^\infty(\pom,\sigma)}\,,
\end{align*}
where $\omega^X$ is the harmonic measure on $\pom$ with pole at $X$. Thus, by Theorem \ref{theorem:eps-approximability}, Lemma \ref{lemma:smooth_approximators} and \eqref{eq5.1a}, 
there exists a smooth $\tfrac{1}{2}$-approximator of $u_0$, i.e.\ a function $\Phi_0 \in  C^\infty(\Omega)$ such that
\begin{align*}
  \|u_0 - \Phi_0\|_{L^\infty(\Omega)} \le \frac{1}{2}\|u_0\|_{L^\infty(\Omega)} \ \ \text{ and } \ \ \sup_{x \in \pom, r > 0}  \frac{1}{r^n} \iint_{B(x,r) \cap \Omega} |\nabla \Phi_0(Y)| \, dY \leq C_0 \|u_0\|_{L^\infty(\Omega)},
\end{align*}
where $C_{0}$ depends only on dimension and the ADR and UR constants for $\pom$.
By  Lemma \ref{lemma:two-sided_traces} and Remark \ref{r4.16}, $\Phi_0$ has a non-tangential trace (in at least a $1$-sided sense),  defined $\sigma$-a.e.\ on $\pom$, that we denote by $\varphi_0$.
Furthermore, 
by Lemma \ref{lemma-trace}, the non-tangential trace $\tmf u_0(x)$ exists, with
\begin{equation}\label{eq6.1a}
\tmf u_0(x) = f_0(x) = f(x)\,,\quad \text{for } \sigma\text{\rm-a.e.\ } x \in \pom\,.
\end{equation} 
Let $Z_1\subset \pom$ denote the set where either $\varphi_0$ does not exist, or where \eqref{eq6.1a} fails, hence $\sigma(Z_1)=0$.  Since $\sigma$ is a Borel regular measure, we may assume without loss of generality that $Z_1$ is a Borel set.  We now define
\begin{align*}
  f_1(x) \coloneqq \left\{ \begin{array}{cl}
                         f_0(x)-\varphi_0(x)&\text{, if } x\in \pom\setminus Z_1 \\[4pt]
                         0 &\text{, if } x\in Z_1\,.
                       \end{array} \right. 
\end{align*}
Then $f_1$ is an everywhere bounded Borel measurable function on $\pom$, so there is a harmonic function
\begin{align*}
  u_1(X) \coloneqq \int_{\pom} f_1(y) \, d\omega^X(y),\quad X\in \Omega\,,
\end{align*}
satisfying
\begin{align*}
  \|u_1\|_{L^\infty(\Omega)}\leq \|f_1\|_{\sup(\pom)} \leq  \|u_0 - \Phi_0\|_{L^\infty(\Omega)} \le 
  \frac{1}{2}\|u_0\|_{L^\infty(\Omega)}
  \leq \frac12 \|f\|_{L^\infty(\pom)}.
\end{align*}
Again using Theorem \ref{theorem:eps-approximability} and Lemma \ref{lemma:smooth_approximators}, we may construct a smooth $\tfrac{1}{2}$-approximator of $u_1$, i.e.\ a function $\Phi_1 \in  C^\infty(\Omega)$ such that
\begin{align*}
  &\|u_1 - \Phi_1\|_{L^\infty(\Omega)} \le \frac{1}{2}\|u_1\|_{L^\infty(\Omega)} \leq \frac14 \|u_0\|_{L^\infty(\Omega)}, \quad \text{ and } \\ 
  \sup_{x \in \pom, r > 0}  \frac{1}{r^n} &\iint_{B(x,r) \cap \Omega} |\nabla \Phi_1(Y)| \, dY 
  \leq C_0 \|u_1\|_{L^\infty(\Omega)} \leq \frac12 C_0 \|u_0\|_{L^\infty(\Omega)}\,,
\end{align*}
with $C_{0}$ as above. By  Lemma \ref{lemma:two-sided_traces} and Remark \ref{r4.16}, $\Phi_1$ has a non-tangential trace (in at least a 1-sided sense), defined $\sigma$-a.e.\ on $\pom$, that we denote by $\varphi_1$. Moreover, by Lemma \ref{lemma-trace}, $u_1$ has a non-tangential trace $\tmf u_1$ such that
\begin{equation}\label{eq6.1} 
\tmf u_1(x) = f_1(x) =f_0(x) -\varphi_0(x)\,, \quad \sigma\text{\rm-a.e.\ } x \in \pom\,.
\end{equation}
Let $Z_2\subset \pom$ be the set of $\sigma$-measure $0$ such that either \eqref{eq6.1} fails, or $\varphi_1$ does not exist. Again, without loss of generality, we may assume that $Z_2$ is a Borel set. We set
\begin{align*}
  f_2(x) \coloneqq \left\{ \begin{array}{cl}
                         f_1(x)-\varphi_1(x)= f_0(x) - \varphi_0(x) -\varphi_1(x)&\text{, if } x\in \pom\setminus Z_2 \\[4pt]
                         0 &\text{, if } x\in Z_2\,.
                       \end{array} \right. 
\end{align*}
We let $u_2$ be the harmonic extension of $f_2$, and iterate, to obtain for each $k \in \N_0$, a sequence of Borel sets $Z_k\subset \pom$ of $\sigma$-measure $0$, harmonic functions $u_k$, their $\tfrac{1}{2}$-approximators $\Phi_k$, the non-tangential boundary traces $\varphi_k$ of the approximators, and the non-tangential boundary traces $f_{k+1}$ of the function $u_k - \Phi_k$. These satisfy
\begin{enumerate}
  \item[(i)] $f_{k+1} = f_0(x) - \sum_{i=0}^k \varphi_i(x)$, $x\in \pom\setminus Z_{k+1}$,
  
  \smallskip
  
  \item[(ii)] $\|f_{k+1}\|_{\sup(\pom)} \le \|u_k - \Phi_k\|_{L^\infty(\Omega)} \leq
 2^{-k-1}\|u_0\|_{L^\infty(\Omega)} \le 2^{-k-1} \|f_0\|_{\sup(\pom)}$,
  
  \smallskip
  
  \item[(iii)] $\|u_k\|_{L^\infty(\Omega)} \leq \|f_{k}\|_{\sup(\partial \Omega)} \le  
  2^{-k} \|u_0\|_{L^\infty(\Omega)}$
  
    \smallskip

  \item[(iv)] $\sup_{x \in \pom, r > 0}  \frac{1}{r^n} \iint_{B(x,r) \cap \Omega} |\nabla \Phi_k(Y)| \, dY \le
  C_0 \|u_k\|_{L^\infty(\Omega)} \le  2^{-k} C_0 \|u_0\|_{L^\infty(\Omega)}$.
  
     \smallskip

  \item[(v)] $ \| \Phi_k\|_{L^\infty(\Omega)} \lesssim  2^{-k} \|u_0\|_{L^\infty(\Omega)}$ (by (ii), (iii) and 
  the triangle inequality).
  
\end{enumerate}
By (v), we may define  the uniformly convergent series
\begin{align}\label{eq5.2}
  \Phi(X) \coloneqq \sum_{k=0}^\infty \Phi_k(X)\,,\quad X\in\Omega\,.
\end{align}
By construction, the function $\Phi$ has a non-tangential boundary trace $\varphi$ 
(in at least a 1-sided sense; we recall that the 1-sided approach may be taken to be the same
for all $\Phi_k$: see Remark \ref{r4.16}), defined $\sigma$-a.e.\ on $\pom$,
\begin{align*}
  \varphi(x) = \sum_{k=0}^\infty \varphi_k(x).
\end{align*}
Since $\lim_{k \to \infty} \|f_k\|_{\sup(\pom)} \le \lim_{k \to \infty} 2^{-k} \|f_0\|_{\sup(\pom)} = 0$, 
we have $\lim_{k \to \infty} f_k(x) = 0$ for 
every $x \in \pom$. In particular,  by (i) above we have
\begin{align*}
  0 = \lim_{k \to \infty} f_{k+1}(x)
    = \lim_{k \to \infty} \left( f_0(x) - \sum_{i=0}^k \varphi_i(x) \right)
    = f_0(x) - \varphi(x)\,,\quad \sigma\text{\rm-a.e.\ } x \in \pom
\end{align*}
(that is, for $x\in \pom \setminus(\cup_k Z_k)$). Thus, $\varphi(x) = f(x)$ for $\sigma$-a.e.\ $x \in \pom$, since $f_0=f$ at $\sigma$-a.e.\ point on $\pom$. Also, for $x\in \pom$ and $r>0$, and for every $\oPsi \in C_0^1(B(x,r)\cap \Omega)$ satisfying $\|\oPsi\|_{L^\infty} \le 1$, using \eqref{eq5.2} and then (iv), we have
\begin{align*} 
 \frac{1}{r^n} \iint_{B(x,r) \cap \Omega} \Phi(Y) \, \text{div} \oPsi(Y) \, dY
  &=  
   \sum_{k=0}^\infty \frac{1}{r^n} \iint_{B(x,r) \cap \Omega} \Phi_k(Y) \, \text{div} \oPsi(Y) \, dY \\
  &\leq  \frac{1}{r^n} \iint_{B(x,r) \cap \Omega} |\nabla \Phi_k(Y)| \, dY\\
  &\le \sum_{k=0}^\infty 2^{-k} C_0 \|u_0\|_{L^\infty(\Omega)}
  = 2C_0 \|u_0\|_{L^\infty(\Omega)}.
\end{align*}
Thus, the measure $\mu \coloneqq |\nabla \Phi(Y)| \, dY$ is a Carleson measure. 

By Lemmas \ref{lemma:smooth_regularization}, \ref{lemma:gradient_regularization}, 
\ref{lemma:carleson_regularization} and \ref{lemma:smooth_nt_convergence}, we may further assume that $\Phi\in  C^\infty(\Omega)$, and that $|\nabla \Phi(X)|\lesssim \|u_0\|_{L^\infty(\Omega)} \delta(X)^{-1}$.  Since $\|u_0\|_{L^\infty(\Omega)}\leq
\|f\|_{L^\infty(\pom)}$, this completes the proof of Theorem \ref{theorem:bounded_extension}.

\begin{remark}
  \label{r6.4}
  Note that the preceeding argument involved the construction of a bounded harmonic extension $u$, corresponding to given Borel measurable data $f\in L^\infty (\pom,d\sigma)$, such that the non-tangential trace $\tmf u$ satisfies $\tmf u(x)= f(x)$ for $\sigma$-a.e.\ $x \in \pom$. It is perhaps worthwhile to observe that, in the absence of absolute continuity of harmonic measure with respect to $\sigma$, this extension need not be unique. Indeed, suppose that $\|f\|_{\sup(\pom\setminus Z)} =\|f\|_{L^\infty(\pom,\sigma)}=1$, for a Borel set $Z\subset\pom$ with $\sigma(Z)=0$. Set
  \begin{align*}
    g_0(x) \coloneqq \left\{ \begin{array}{cl}
                         f(x)&\text{, if } x\in \pom\setminus Z \\[4pt]
                         0 &\text{, if } x\in Z\,,
                       \end{array} \right.  \, \quad
                        g_1(x) \coloneqq \left\{ \begin{array}{cl}
                         f(x)&\text{, if } x\in \pom\setminus Z \\[4pt]
                         1 &\text{, if } x\in Z\,,
                       \end{array} \right.
  \end{align*}
  and define
  \begin{align*}
    v_i(Y) \coloneqq v_{g_i}(Y) \coloneqq \int_{\pom} g_i\, d\omega^Y\,,\quad  Y\in \Omega\,,\, i=0,1\,.
  \end{align*}
  Then $\|v_i\|_{L^\infty(\Omega)} \leq 1$ for $i=0,1$, and by Lemma \ref{lemma-trace}, the traces $\tmf v_0$ and $\tmf v_1$ exist $\sigma$-a.e.\ on $\pom$, and satisfy
  \begin{align*}
    \tmf v_0 = g_0 = f = g_1 = \tmf v_1\,,\quad \sigma\text{\rm-a.e.\ on }\pom\,.
  \end{align*}
  On the other hand,
  \begin{align*}
    v_1(Y) = v_0(Y) +  \omega^Y(Z)\,,
  \end{align*}
  so if harmonic measure has positive mass on $Z$, then $v_1 \neq v_0$.
\end{remark}

\section{Carleson boxes, Carleson tents and Whitney regions}
\label{section:carleson_boxes}

Before we prove Proposition \ref{proposition:dyadic_extension}, we revisit the construction of Whitney regions and Carleson boxes.  The previous construction (see Subsection \ref{section:whitney_cubes}, and \cite[Section 3]{hofmannmartellmayboroda}) is not suitable for our current purposes, since the overlap of the Whitney and Carleson regions 
causes technical difficulties related to the Carleson measure estimates. 

Since we do not need many of the strong geometric properties of the Carleson boxes constructed in \cite{hofmannmartellmayboroda}, we start by presenting a simplified construction of the boxes and proving that the boundaries of the boxes inside $\Omega$ are upper $n$-ADR. We note that the original proof for the upper $n$-ADR property of the boundaries of Carleson boxes in \cite[Appendix]{hofmannmartellmayboroda} does not apply ``off-the-shelf" in our situation because we do not use dilated (hence overlapping)
Whitney cubes (as is done in \cite[Appendix]{hofmannmartellmayboroda}). However, our approach makes the proof quite simple.

In this section, $\Omega \subset \R^{n+1}$ is an open set, satisfying the corkscrew condition, with $d$-ADR boundary $\pom$ for some $d \in (0,n]$, and $\D$ is a dyadic system on $\pom$. Recall the Whitney decomposition and the definition of the collections $\Wc_Q = \Wc_Q(\eta,K)$ from Subsection \ref{section:whitney_cubes}. 

\begin{remark}
  In this and the next two sections, it will be technically convenient to work with ``half-open" Whitney cubes, that is, in Sections \ref{section:carleson_boxes}, \ref{section:carleson_regions}, and \ref{section:dyadic_extension}, a cube $I \in \Wc$ is assumed to be of the form $I =\Pi_{k=1}^{n+1}(a_k,a_k+h]$, with $\ell(I) = h \approx \dist(I,\pom)$. All other properties of the Whitney cubes will be exactly as before.
\end{remark}

We start by noting that our Whitney regions are not empty:
\begin{lemma}
  \label{lemma:constant_choice}
  We can choose the parameters $\eta$ and $K$ depending only on
   the corkscrew constants, 
   so that $\Wc_Q \neq \emptyset$ for every $Q \in \D$.
\end{lemma}
The proof is a straightforward generalization of \cite[Remark 3.3]{hofmannmartellmayboroda} and \cite[Lemma 5.3]{hofmannmartell}. We omit the details.

Let us remark that in the codimension 1 case, if $\Omega =\ree\setminus E$, with $E$ $n$-ADR, then the corkscrew condition holds automatically, with constants that in turn depend only on dimension and ADR.  Moreover, in the $d$-ADR case with $d<n$,  $\Omega =\ree\setminus E$ has only one connected component, which necessarily satisfies the corkscrew condition.

\begin{defin}
  \label{defin:dyadic_regions_and_cones}
  Suppose that $x \in \partial \Omega$ and $Q \in \D$. The \emph{``half-open" Whitney region relative to $Q$} is the set
  \begin{align*}
    U_Q \coloneqq \bigcup_{I \in \Wc_Q} I,
  \end{align*}
  the \emph{dyadic cone at $x$} is the set
  \begin{align*}
    \Gamma(x) \coloneqq \bigcup_{Q' \in \D: \, x \in Q'} U_{Q'} 
  \end{align*}  
  the  \emph{Carleson box relative to $Q$} is the set
  \begin{align*}
    T_Q \coloneqq \bigcup_{Q' \in \D, Q' \subseteq Q} U_{Q'}
  \end{align*}
  and the  \emph{Carleson tent relative to $Q$} is the set
  \begin{align*}
    \tau_Q \coloneqq  \Omega \setminus \bigcup_{y \in \partial \Omega \setminus Q} \Gamma(y)
  \end{align*}
\end{defin}

\begin{figure}[ht]
  \includegraphics[scale=0.55]{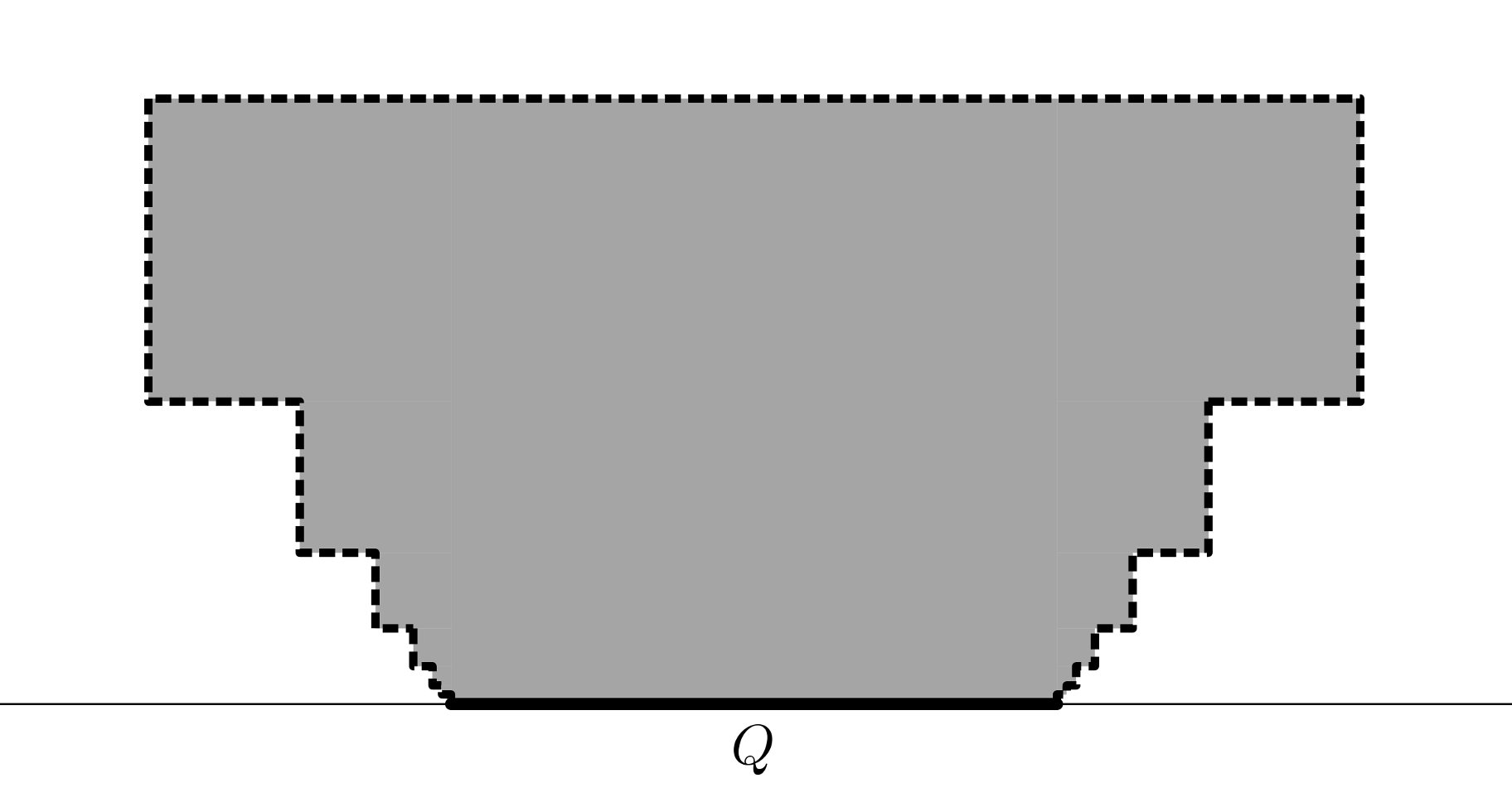} \ \ \ \ \includegraphics[scale=0.55]{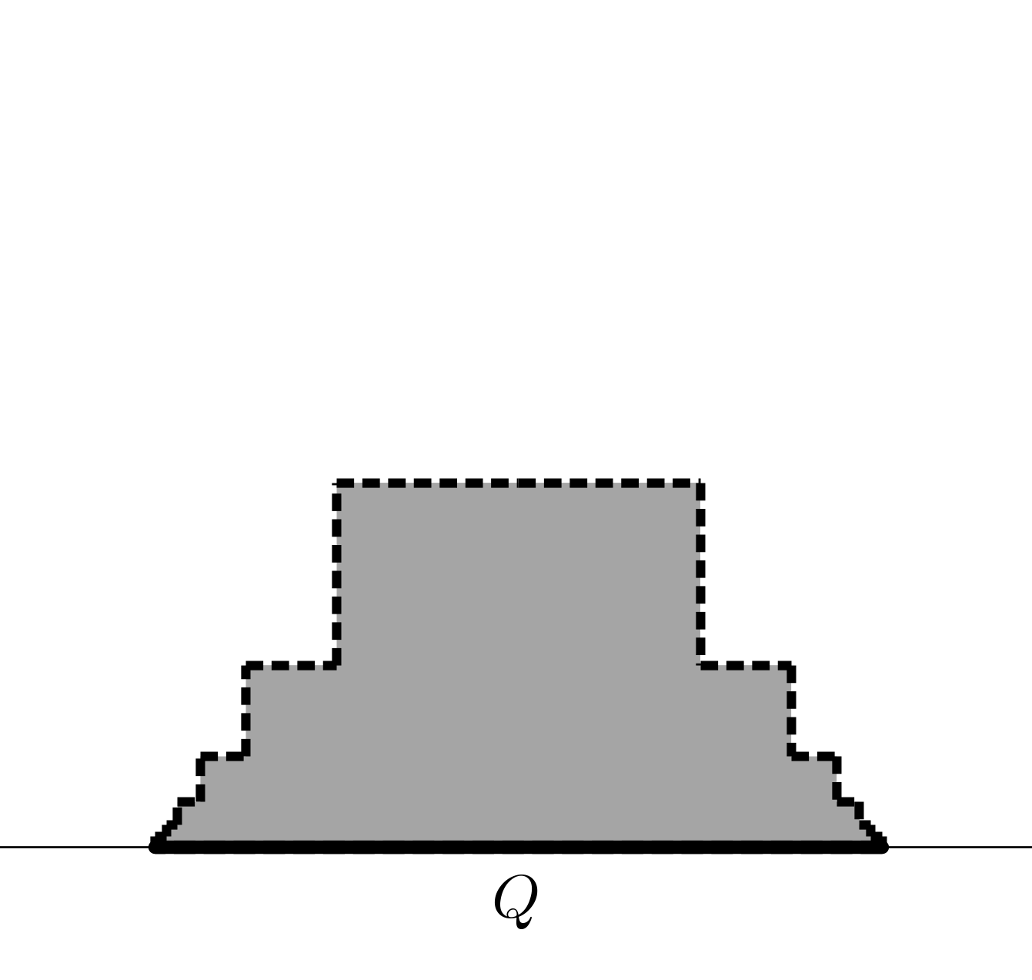}
  \caption{A rough idea of the structure of $T_Q$ (left) and $\tau_Q$ (right) on top of a same cube $Q$ in the simplest case where $\Omega = \R^2_+$.}
\end{figure}

\begin{remark}
  \label{r7.3}
  We note that every $I \in \Wc$ with $\ell(I) \lesssim \diam(\pom)$ belongs to the collection $\Wc_{Q_I}$, where as above $\ell(Q_I) = \ell(I) \approx \dist(I,Q_I)$, and $Q_I$ is chosen to minimize $\dist(I,Q_I)$. Moreover, for $\eta$ chosen small enough and $K$ large enough depending only on the properties of the Whitney decomposition, every $J\in \Wc$ whose closure touches the closure of $I$, also belongs to $\Wc_{Q_I}$.  Consequently, for such $\eta$ and $K$, we have:
  \begin{enumerate}
    \item[$\bullet$] if $\diam(\pom) < \infty$ and $\diam(\Omega) = \infty$, then $\bigcup_{Q \in \D} T_Q \supset B(x,R) \cap\Omega$ for some point $x \in \pom$ and $R \approx\diam(\pom)$,
  
    \item[$\bullet$] if $\diam(\pom) \approx \diam(\Omega)$, then $\bigcup_{Q \in \D} T_Q \supset \Omega$. 
  \end{enumerate}
\end{remark}

\begin{remark}
  \label{r7.4}
  Given $m\in (1,\infty)$, one may choose $\eta$ small enough and $K$ large enough, depending on $m$, so that the dyadic cone $\Gamma(x)$ contains (at least locally) a cone of the type $\widetilde{\Gamma}^m(x) = \{Y \in \Omega \colon \dist(x,Y) < m \delta(Y)\}$; i.e., $\widetilde{\Gamma}^m(x) \cap B(x,R) \subset \Gamma(x)$ for $R \approx \diam(\pom)$. We omit the routine proof of this fact.
\end{remark}

We now fix a suitably large aperture constant $m$ that allows us to apply Lemma \ref{lemma:smooth_nt_convergence} later. Combining Lemma \ref{lemma:constant_choice} and Remarks \ref{r7.3} and \ref{r7.4}, we see that we may (and do) choose $\eta$ and $K$ depending only on the corkscrew constants, the Whitney cube constants, and the fixed aperture parameter $m$,  in such a way that the collections $\Wc_Q$ are non-empty, the Carleson boxes $T_Q$ have good covering properties and the dyadic cones contain ``regular'' cones. The sets $U_Q$, $T_Q$ and $\Gamma(x)$ then satisfy the same properties (with possibly different implicit constants) as $\Uc_Q$, $\Tc_Q$ and $\Upsilon_Q(x)$ in Lemma \ref{lemma:properties_of_whitney_regions}, excluding naturally the last two properties related to the bilateral corona decomposition.

Next we prove that the boundaries of the boxes $T_Q$ in $\Omega$ are upper $n$-ADR. The boundaries of the boxes constructed in \cite{hofmannmartellmayboroda} are also lower $n$-ADR, but for our present purposes we shall need only the upper $n$-ADR property. We first prove a preliminary lemma, which will also be useful in the sequel.

\begin{lemma}
  \label{l7.6}
  Let $Q\in\D$.  Then for each positive $\kappa <\infty$
  \begin{equation}
    \label{eq7.7} \sum_{\substack{Q' \in \D_Q  \\ \dist(Q',Q^c)\leq \kappa \ell(Q')}}\sum_{I \in \Wc_{Q'}} 
    \Hc^n(\partial I) \leq C_\kappa \,\ell(Q)^n.
  \end{equation}
\end{lemma}

\begin{proof}
  Note that the number of Whitney cubes in $\Wc_{Q'}$ is uniformly bounded for each $Q'$, and that for $I\in \Wc_{Q'}$ we have $\Hc^n(\partial I) \approx \ell(Q')^n$, by 
  the definition of $\Wc_{Q'}$; consequently
  \begin{align*}
    \sum_{I \in \Wc_{Q'}} \Hc^n(\partial I) \lesssim \ell(Q')^n \,.
  \end{align*}
  Organizing the subcubes of $Q$ by dyadic generation $\D_Q = \cup_{k=0}^{\infty}\D^k_Q$, where
  \begin{align*}
    \D^k_Q \coloneqq \{Q'\subset Q:\, \ell(Q') = 2^{-k} \ell(Q)\}\,,\quad 0\leq k\leq \infty\,,
  \end{align*}
  we obtain by the thin boundary property (Theorem \ref{theorem:existence_of_dyadic_cubes} (v)) that
  \begin{equation}
    \label{eq7.6}
    \sum_{\substack{Q' \in \D^k_Q \\ \dist(Q',Q^c)\lesssim\, \ell(Q')}} \sigma(Q') \lesssim 2^{-k\gamma} \sigma(Q) \,.
  \end{equation}
  Combining these observations, we obtain in the codimension 1 case $d=n$ that
   \begin{equation*}
    \sum_{k=0}^\infty \sum_{\substack{Q' \in \D^k_Q\\  \dist(Q',Q^c)\lesssim \,\ell(Q')}}
    \sum_{I \in \Wc_{Q'}} \Hc^n(\partial I) 
    \lesssim \sum_{k=0}^\infty  2^{-k\gamma} \sigma(Q) \lesssim \sigma(Q)\,,
    \end{equation*}
 or in general that 
  \begin{align*}
    \sum_{k=0}^\infty \sum_{\substack{Q' \in \D^k_Q \\ \dist(Q',Q^c)\lesssim \,\ell(Q')}}
    \sum_{I \in \Wc_{Q'}} \Hc^n(\partial I) 
    &\lesssim \sum_{k=0}^\infty \sum_{\substack{Q' \in \D^k_Q \\ \dist(Q',Q^c)\lesssim\, \ell(Q')}} \ell(Q')^n \\
    &\le \ell(Q)^{n-d} \sum_{k=0}^\infty \sum_{\substack{Q' \in \D^k_Q \\ \dist(Q',Q^c)\lesssim \ell(Q')}} \ell(Q')^d \\
    &\approx \ell(Q)^{n-d} \sum_{k=0}^\infty \sum_{\substack{Q' \in \D^k_Q \\ \dist(Q',Q^c)\lesssim \ell(Q')}} \sigma(Q') \\
    &\lesssim \ell(Q)^{n-d} \sum_{k=0}^\infty   2^{-k\gamma} \sigma(Q) \lesssim \ell(Q)^n \,.
  \end{align*}
\end{proof}

\begin{lemma}
  \label{lemma:carleson_box_upper_adr}
  For each $Q$, the set $\wip T_Q$ is upper $n$-ADR, where $\wip T_Q \coloneqq \partial T_Q \cap \Omega$: for every $X \in \wip T_Q$ and every $R \in (0,\diam(T_Q))$ we have
  \begin{align*}
    \Hc^n(\wip T_Q \cap B(X,R)) \lesssim R^n,
  \end{align*}
  where the implicit constant depends only on $n$, the ADR constant, the corkscrew constant, the Whitney constants, and the fixed aperture parameter $m$.
\end{lemma}

\begin{proof}
  Note that if $X \in \wip T_Q$, then by construction there exists a dyadic cube $Q' \in \D_Q$ and a Whitney cube $I \in \Wc_{Q'}$ such that $X \in \partial I$. Also, if $\ell(Q') \ll \ell(Q)$ and $\dist(Q',Q^c) \gg \ell(Q')$ for $Q' \in \D_Q$, then $\partial I \cap \wip T_Q = \emptyset$ for every $I \in \Wc_{Q'}$. Thus, if $I\subset T_Q$, with $\partial I \cap \wip T_Q \neq \emptyset$, then $I \in \Wc_{Q'}$ for a cube $Q' \in \D_Q$ such that $\dist(Q',Q^c) \lesssim \ell(Q')$, where the implicit constant depend on $\eta$ and $K$ (which, in turn, we have chosen to depend only on  the corkscrew constants, the Whitney constants, and $m$).   
  
  Consequently, using Lemma \ref{l7.6}, we obtain
  \begin{equation*}
    \Hc^n(\wip T_Q)
    \le \sum_{\substack{Q' \in \D_Q  \\ \dist(Q',Q^c)
    \lesssim \ell(Q')}}\sum_{I \in \Wc_{Q'}} \Hc^n(\partial I) 
    \lesssim \ell(Q)^n \,.
  \end{equation*}
  Thus, we have $\Hc^n(\wip T_Q) \lesssim \ell(Q)^n \approx \diam(Q)^n$ for any $Q \in \D$. Let us then prove the upper $n$-ADR property. Suppose that $X \in \wip T_Q$ and $R \in (0,\diam(T_Q))$.  There are three cases:
  \begin{enumerate}
    \item[1)] Suppose that $R \approx \diam(T_Q)$. Then, by the consideration above, we have
              \begin{align*}
                \Hc^n(\wip T_Q \cap B(X,R)) \le \Hc^n(\wip T_Q) \lesssim \diam(Q)^n \approx \diam(T_Q)^n \approx R^n.
              \end{align*}
    
    \item[2)] Suppose that $R \ll \delta(X)$. Then, by construction, $B(X,R) \cap \wip T_Q$ is contained in a union of a uniformly bounded number of boundaries of Whitney cubes $I$ such that $\ell(I) > R$. Since  $\partial I$ is clearly $n$-ADR for each $I\in\Wc$, we therefore find that $\Hc^n(\wip T_Q \cap B(X,R)) \lesssim R^n$.
              
    \item[3)] Suppose that $\delta(X) \lesssim R \ll \diam(T_Q)$. Then $\wip T_{Q} \cap B(X,R) = \wip T_{Q'} \cap B(X,R)$ for some subcube of $Q' \in \D_Q$ with $\ell(Q') \approx R$. Thus, by the consideration above, we have
              \begin{align*}
                \Hc^n(\wip T_{Q} \cap B(X,R)) = \Hc^n(\wip T_{Q'} \cap B(X,R))
                                              \le \Hc^n(\wip T_{Q'})
                                              \lesssim \ell(Q')^n
                                              \approx R^n.
              \end{align*}
  \end{enumerate}
  This completes the proof.
\end{proof}

\section{Modified Carleson tents}
\label{section:carleson_regions}

Fix a cube $Q_0 \in \D$.  For all $Q \subseteq Q_0$, we shall now construct disjoint Carleson tents $t_Q$, that have better covering properties than $\tau_Q$. We let $\{Q_0\}$ be ``generation zero", and then enumerate the dyadic descendants of $Q_0$: let $\{Q_1^i\}_i$ be the first generation of descendants, $\{Q_2^i\}_i$ the second generation of descendants, and so on. Let the number of descendants of generation $k$ be $N(k)$. We construct a restricted version of the Whitney collection $\Wc_Q$, $Q\subset Q_0$,  by removing some of the cubes from $\Wc_{Q_k^i}$: for each $k,i \in \N$, $i \le N(k)$, we set
\begin{align*}
  \Wc^{\text{r}}_{Q_k^i} \coloneqq \Wc_{Q_k^i} 
  \setminus \left( \bigcup_{m=0}^{k-1} \bigcup_{j=1}^{N(m)} \Wc_{Q_m^j} \ \cup \ \bigcup_{j=1}^{i-1} 
  \Wc_{Q_k^j} \right),
\end{align*}
where of course the second union is vacuous if $i=1$, and both are vacuous if $k=0$. Note that the restricted Whitney collections $\{ \Wc^{\text{r}}_Q\}_{Q\subset Q_0}$ are pairwise disjoint, by construction.

We can then define \emph{restricted Whitney regions} $U_Q^{\text{r}}$ and \emph{modified Carleson tents} $t_Q$ for cubes $Q \subseteq Q_0$:
\begin{align}
  \label{defin:modified_tent} U_Q^{\text{r}} \coloneqq \bigcup_{I \in \Wc^{\text{r}}_Q} I, \ \ \ \ \ 
  t_Q \coloneqq \bigcup_{Q' \in \D, Q' \subseteq Q} U_{Q'}^{\text{r}}
\end{align}

\begin{figure}[ht]
  \includegraphics[scale=0.75]{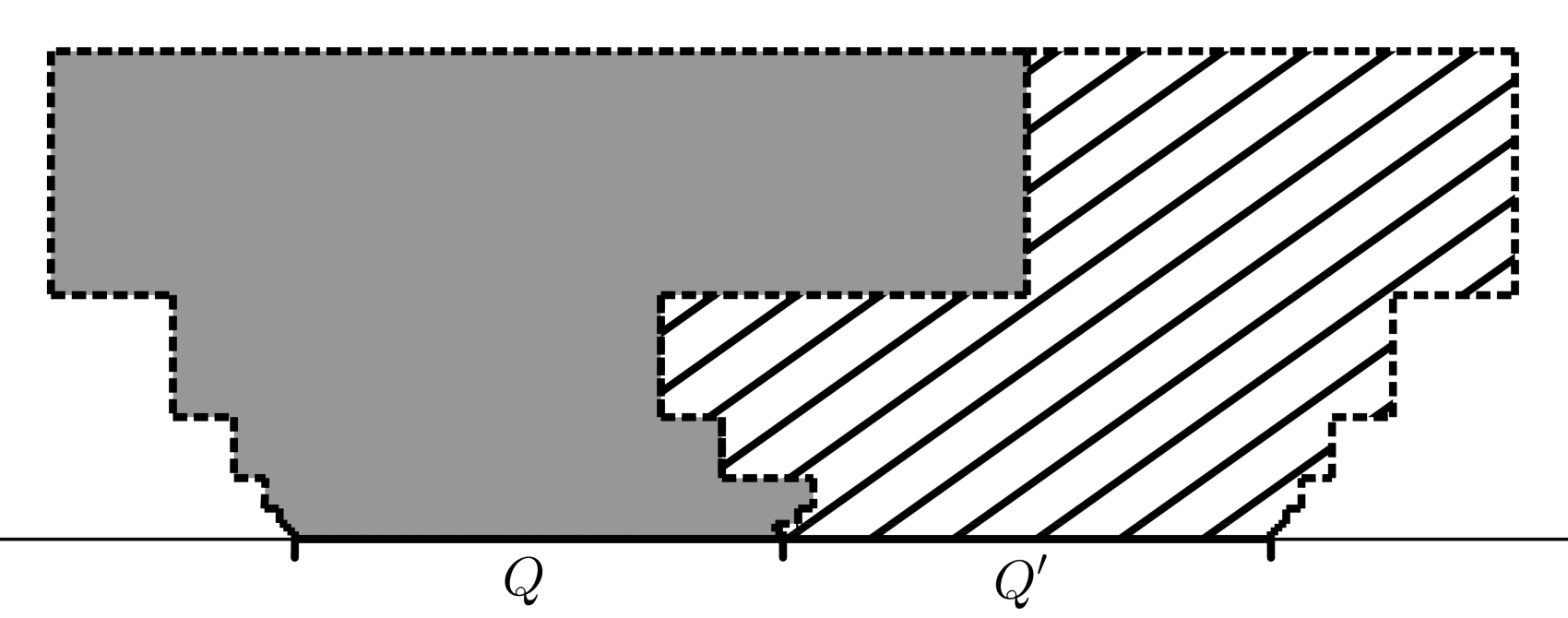}
  \caption{Two modified Carleson tents $t_{Q}$ and $t_{Q'}$ in the simplest case where $\Omega = \R^2_+$. The boundary they share may be slightly messy but it consists of a union of faces of Whitney cubes.}
\end{figure}

\begin{remark}
  \label{r8.2}
  Since the Whitney collections $\{ \Wc^{\text{r}}_Q\}_{Q\subset Q_0}$ are pairwise disjoint, and since we are now working with half-open (hence disjoint) Whitney cubes $I$, it follows that the sets $\{U^{\text{r}}_Q\}_{Q\subset Q_0}$ are also pairwise disjoint.
\end{remark}

\begin{lemma}
  \label{lemma:modified_tents_properties}
  Suppose that $Q,Q_1,Q_2 \in \D_{Q_0}$.  We then have:
  \begin{enumerate}
    \item[i)] $\tau_Q \subset t_Q$.
    
    \item[ii)] If $Q_1 \cap Q_2 = \emptyset$, then also $t_{Q_1} \cap t_{Q_2} = \emptyset$.
    
    \item[iii)] If $Q_1 \subset Q_2$, then also $t_{Q_1} \subset t_{Q_2}$.
    
    \item[iv)] $T_{Q_0} = t_{Q_0}$.
   Moreover, for $Q\subsetneq Q_0$, there is a collection $\Fc(Q)=\{Q^i\}_{i=1}^N \subset \D_{Q_0}$,
  of uniformly bounded cardinality $N$ depending only on $n$, ADR, $\eta$ and $K$, such that 
  $\ell(Q^i) \approx_{\eta,K} \ell(Q)$ with
               $\ell(Q^i) = \ell(Q^{i'})$ for all $i, i'$, and $T_Q \subset \bigcup_i t_{Q^i}$. 
  \end{enumerate}
\end{lemma}

\begin{proof}
  The properties ii), iii), and iv) follow directly from the construction so we prove only property i).

  Note that by construction (see Definition \ref{defin:dyadic_regions_and_cones}),
  \begin{align*}
    \bigcup_{y \in \pom\setminus Q}\Gamma(y) = \bigcup_{Q'\in \D\setminus\D_Q}U_{Q'}\,,
  \end{align*}
  and that $U_{Q'}^{\text{r}} \subset U_{Q'}$ for every $Q' \in \D_{Q_0}$. Moreover, the restricted Whitney regions $U_{Q'}^{\text{r}}$ are disjoint (see Remark \ref{r8.2}).  Consequently, 
  \begin{align*}
    \tau_Q = \Omega \setminus \bigcup_{y \in \partial \Omega \setminus Q} \Gamma(y)
           = T_Q \setminus \bigcup_{Q' \in \D \setminus \D_Q} U_{Q'}  
           \subset T_Q \setminus \bigcup_{Q' \in \D_{Q_0} \setminus \D_{Q}} U_{Q'}^\text{r} 
           \subset t_Q.
  \end{align*}
\end{proof}

\begin{lemma}
  \label{lemma:adr_modified_tents}
  The sets $\partial t_Q \cap \Omega$ are upper $n$-ADR with the ADR constant depending only on the dimension and the ADR constant of $\partial \Omega$.
\end{lemma}

\begin{proof}
  Recall that $\tau_Q\subset t_Q$, by Lemma \ref{lemma:modified_tents_properties} i). Thus, if $I\subset t_Q$, 
  with $\partial I \cap \partial t_Q \neq \emptyset$, then $I \in \Wc^{\text{r}}_{Q'}$ for a cube $Q' \in \D_Q$ such that $\dist(Q',Q^c) \lesssim \ell(Q')$. One may then use Lemma \ref{l7.6}, following the proof of Lemma \ref{lemma:carleson_box_upper_adr} with minor adjustments.  We omit the details.
\end{proof}

\section{Proof of Proposition \ref{proposition:dyadic_extension}}
\label{section:dyadic_extension}

Suppose that $\Omega \subset \R^{n+1}$ is an open set satisfying the corkscrew condition with $d$-ADR boundary for some $d \in (0,n]$. Let $Q_0 \in \D$ be a fixed dyadic cube, $\widetilde{\D}_{Q_0} = \{Q_j\}_j \subset \D_{Q_0}$ be a collection of subcubes of $Q_0$ and $\{\alpha_j\}_j$ a collection of coefficients such that 
\begin{align*}
  f(x) \coloneqq \sum_j \alpha_j 1_{Q_j},
\end{align*}
belongs to $\BMO(\pom)$, the collection $\widetilde{\D}_{Q_0}$ enjoys a Carleson packing condition with packing norm $\Cs_{\widetilde{\D}_{Q_0}} \eqqcolon C_0$ (see Definition \ref{defin:carleson_packing_norm}), and $\sup_j |\alpha_j| \lesssim \|f\|_{\BMO}$.  Note that $f$ vanishes on $\pom \setminus Q_0$, but we assume that $f\in$ BMO, globally on $\pom$.
We denote
\begin{align*}
  F_0 \coloneqq \sum_j \alpha_j 1_{t_{Q_j}},
\end{align*}
where $t_{Q_j}$ is the modified Carleson tent defined in \eqref{defin:modified_tent}. We will show that a smooth version of $F_0$ satisfies the properties in Proposition \ref{proposition:dyadic_extension}.

We start by proving the following estimate that we shall need later:
\begin{lemma}
  \label{lemma:bounded_sum_difference}
  Let $Q, Q' \in \D$ be such that
  \begin{align}
    \label{assumption:cubes} \ell(Q) \approx \ell(Q') \gtrsim \dist(Q,Q').
  \end{align}
  Then 
  \begin{align*}
    \left| \sum_{j: \, Q_j\supseteq Q} \alpha_j - \sum_{j: \, Q_j\supseteq Q'} \alpha_j \right| \lesssim C_0\|f\|_{\BMO},
  \end{align*}
  where the implicit constant depends on the implicit constant in \eqref{assumption:cubes}.
\end{lemma}

\begin{proof}
  Let us fix two disjoint cubes $Q, Q' \in \D$, that satisfy \eqref{assumption:cubes}. Fix a constant $C$ large enough 
  (depending only on the implicit constants in \eqref{assumption:cubes}) that $Q \cup Q' \subset B_Q^{*} \coloneqq B(x_Q,r)$, with $r \coloneqq C\,\ell(Q)$. Let $\Delta_Q^{*} \coloneqq B_Q^{*}\cap\pom$ denote the corresponding surface ball. Since $f \in \BMO(\pom)$, by the ADR property we have
  \begin{equation}
    \label{eq2}
    \fint_Q |f - \langle f \rangle_{\Delta_Q^*}| + \fint_{Q'} |f - \langle f \rangle_{\Delta_Q^*}| 
    \lesssim \fint_{\Delta_Q^*} |f - \langle f \rangle_{\Delta_Q^*}|
    \le \|f\|_{\BMO}\,.
  \end{equation}
  By the uniform bound on the coefficients and the packing condition of the collection $\{Q_j\}_j$, we have that
  \begin{align*}
    \fint_Q\, \left|\sum_{j: \, Q_j\subseteq Q} \alpha_j 1_{Q_j}(x)\right|  dx 
    \le \frac{\sup_j |\alpha_j|}{\sigma(Q)} \sum_{j: Q_j \subseteq Q} \sigma(Q_j)
    \le C_0 \|f\|_{\BMO},
  \end{align*}
  and similarly with $Q'$ in place of $Q$. Combining this observation with \eqref{eq2}, we see that
  \begin{align*}
    \left|\sum_{Q_j: \, Q\subsetneq Q_j} \alpha_j - \langle f \rangle_{\Delta_Q^*} \right|
    &= \fint_Q\,\left|\sum_{Q_j: \, Q\subsetneq Q_j} \alpha_j 1_{Q_j}(x) - \langle f \rangle_{\Delta_Q^*} \right| dx \\
    &= \fint_Q\,\left|f(x) - \langle f \rangle_{\Delta_Q^*} - \sum_{j: \, Q_j \subseteq Q} \alpha_j 1_{Q_j}(x) \right| dx
    \lesssim C_0 \|f\|_{\BMO},
  \end{align*}
  and similarly
  \begin{align*}
    \left|\sum_{Q_j: \, Q'\subsetneq Q_j} \alpha_j - \langle f \rangle_{\Delta_Q^*} \right|  \lesssim C_0 \|f\|_{\BMO}.
  \end{align*}
  By the triangle inequality, these last two estimates yield
  \begin{align*}
    \left|\sum_{Q_j: Q\subsetneq Q_j} \alpha_j  - \sum_{Q_j: Q'\subsetneq Q_j} \alpha_j\right| \lesssim C_0\|f\|_{\BMO}.
  \end{align*}
\end{proof}

\begin{lemma}
  \label{lemma:dyadic_extension_nt}
  We have
  \begin{align*}
    \lim_{Y \to x \text{ N.T.}} F_0(Y) = f(x)
  \end{align*}
  for $\sigma$-a.e.\ $x \in \pom$. Here $\lim_{Y \to x \text{ N.T.}}$ stands for standard type non-tangential convergence.
\end{lemma}

\begin{proof}
  By the Carleson packing condition of $\widetilde{\D}_{Q_0}$, and the uniform boundedness of the coefficients $\alpha_j$, it follows that $\sum_{j} 1_{Q_j}(x) < \infty$, and hence also
  $|\sum_{j} \alpha_j 1_{Q_j}(x) |< \infty$, for $\sigma$-a.e.\ $x \in \pom$. Also, $\sum_j \alpha_j 1_{t_{Q_j}}(Y) < \infty$ for each $Y \in \Omega$, since $Y$ can belong to only a finite number of modified tents $t_{Q_j}$ (those for which $\ell(Q_0)\geq \ell(Q_j)\gtrsim \delta(Y))$. Thus,
  \begin{equation*}
    \widetilde{F}_0(x,Y) \coloneqq \sum_j \alpha_j 1_{Q_j}(x) - \sum_j \alpha_j 1_{t_{Q_j}}(Y)
    = \sum_j \alpha_j \left( 1_{Q_j}(x) - 1_{t_{Q_j}}(Y) \right)
  \end{equation*}
  is absolutely convergent for $\sigma$-a.e.\ $x \in \pom$, and all $Y \in \Omega$. 
  For fixed $x$ with $\sum_j  1_{Q_j}(x) < \infty$, we split
  \begin{align*}
    \bigcup_j Q_j = \Big( \bigcup_{\Fc_1(x)} Q_j\Big) \bigcup \Big( \bigcup_{\Fc_2(x)} Q_j \Big),
  \end{align*}
  where $\Fc_1(x) \coloneqq \{Q_j\in \widetilde{\D}_{Q_0} \colon x\in Q_j\}$, and 
  $\Fc_2(x) \coloneqq  \{Q_j\in \widetilde{\D}_{Q_0} \colon x\in \pom\setminus Q_j\}$.  
  In turn,
  \begin{align*}
    \widetilde{F}_0(x,Y)
    &= \sum_{Q_j \in \Fc_1(x)} \alpha_j \left( 1_{Q_j}(x) - 1_{t_{Q_j}}(Y) \right)
    + \sum_{Q_j \in \Fc_2(x)} \alpha_j \left( 1_{Q_j}(x) - 1_{t_{Q_j}}(Y) \right) \\[4pt] &\eqqcolon \widetilde{F}^1_0(x,Y) + \widetilde{F}^2_0(x,Y)\,. 
  \end{align*}
  In particular, for $x \in \pom \setminus Q_0$, we have $\Fc_2(x) = \widetilde{\D}_{Q_0}$, and $\widetilde{F}^2_0(x,Y) = \widetilde{F}_0(x,Y)$, since $Q_j\subset Q_0$ for each $j$.
    
  Let us then show that $\lim_{Y \to x \text{ N.T.}} \widetilde{F}^i_0(x,Y) = 0$, $i=1,2$, for almost every $x$. Suppose that $\eps > 0$, $Y_\eps \in \Gamma(x)$ and $\dist(x,Y_\eps) < \eps$.  For those $j$ such that $x \in Q_j$, we have $1_{Q_j}(x) - 1_{t_{Q_j}}(Y_\eps) \neq 0$ only if $Y_\eps \in \Gamma(x) \setminus t_{Q_j}$. Thus,
  \begin{align*}
    |\widetilde{F}^1_0(x,Y_\eps)| 
    &\le \sup_j |\alpha_j| \Big( \sum_{\ell(Q_j) \le \sqrt{\eps}} 1_{Q_j}(x) \,1_{ \Gamma(x) \setminus t_{Q_j}}(Y_\eps) 
      \,+ \, \sum_{ \ell(Q_j) > \sqrt{\eps}} 1_{Q_j}(x) \,1_{ \Gamma(x) \setminus t_{Q_j}}(Y_\eps) \Big) \\[4pt]
    &\eqqcolon \sup_j |\alpha_j| \big( I_1^\eps(x) + I_2^\eps(x) \big),
  \end{align*}
  where $\sup_j |\alpha_j| \lesssim \|f\|_{\BMO}$ by assumption. Recall that we have fixed $x$ with $\sum_j  1_{Q_j}(x) < \infty$.  Thus, $I_1^{\eps}(x) \leq \sum_{\ell(Q_j) \le \sqrt{\eps}} 1_{Q_j}(x)$ is the tail of a convergent series, so that $I_1^\eps \to 0$ as $\eps \to 0$.

  Turning now to $I_2^\eps$, we first note that since $Y_\eps \in \Gamma(x) \setminus t_{Q_j}$, there exists a cube $Q\ni x$, such that $Y_\eps \in U_Q \setminus t_{Q_j}$, with $\ell(Q) \approx \delta(Y_\eps) \lesssim \eps$ for some uniformly bounded implicit constants. If $\eps$ is small enough, then $\ell(Q) \ll \ell(Q_j)$ and thus $Q \subset Q_j$, since $x\in Q\cap Q_j$. Hence also $t_Q \subset t_{Q_j}$. Consequently, $Y_\eps \notin t_Q$, and therefore there exists another cube $Q'$ such that $\ell(Q') \approx \ell(Q)$, $Y_\eps \in t_{Q'}$, and $Q' \cap Q_j = \emptyset$. In particular, $\dist(x,Q_j^c) \lesssim \eps \le \sqrt{\eps} \ell(Q_j)$. We set
  \begin{align*}
    \Sigma_j^\eps \coloneqq \{z \in Q_j \colon \dist(z,Q_j^c) \lesssim \sqrt{\eps} \ell(Q_j)\}
  \end{align*}
  for the same implicit uniform constant as above, and assume that $\eps$ is so small that this constant times $\sqrt{\eps}$ is a lot smaller that $1$. We then have
  \begin{align*}
    I_2^\eps(x) \le \sum_j 1_{\Sigma_j^\eps}(x) \eqqcolon h_\eps(x).
  \end{align*}
  In particular, by \eqref{dyadic_cubes:small_boundaries} and the Carleson packing condition of $\{Q_j\}_j$ we obtain
  \begin{align*}
    \|h_\eps \|_{L^1(Q_0)}
    \le \sum_{Q_j \subset Q_0} \sigma(\Sigma_j^\eps)
    \lesssim \eps^{\gamma} \sum_{Q_j \subset Q_0} \sigma(Q_j)
    \lesssim \eps^\gamma \sigma(Q_0) \to 0\,, \ \ \ \ \text{ as } \eps \to 0.
  \end{align*}
  Thus, there is a sequence $(\eps_k)_k$ with $\eps_k \to 0$ as $k \to \infty$, such that $h_{\eps_{k}}(x) \to 0$ as $k \to \infty$, for $\sigma$-a.e.\ $x \in \pom$. Since $h_\eps$ is pointwise decreasing as $\eps \searrow 0$, we therefore have $h_\eps(x) \to 0$ as $\eps \to 0$, and hence also $\lim_{\eps\to 0}I_2^\eps(x)= 0$, for $\sigma$-a.e.\ $x \in \pom$. 

  Consider now those $j$ such that $x \notin Q_j$. Then $1_{Q_j}(x) - 1_{t_{Q_j}}(Y_\eps) \neq 0$ only if $Y_\eps \in \Gamma(x) \cap t_{Q_j}$. In particular,
  \begin{align*}
    |\widetilde{F}^2_0(x,Y_\eps)|
    &\le \sup_j |\alpha_j| \Big( \sum_{\ell(Q_j)
    \le \sqrt{\eps}} 1_{\pom\setminus Q_j}(x)\,1_{t_{Q_j}\cap\Gamma(x)}(Y_\eps) + \sum_{ \ell(Q_j) > \sqrt{\eps}} 1_{\pom\setminus Q_j}(x)\, 1_{t_{Q_j}\cap\Gamma(x)}(Y_\eps) \Big) \\[4pt]
    &\eqqcolon \sup_j |\alpha_j| \big( J_1^\eps(x) + J_2^\eps(x) \big),
  \end{align*}
  where, as before, $\sup_j |\alpha_j| \lesssim \|f\|_{\BMO}$ by assumption.
               
  For $J_1^\eps$, we first note that  $Y_\eps \in t_{Q_j}$ implies $\delta(Y_\eps) \leq \dist(Y_\eps,Q_j)\lesssim \ell(Q_j)$.  Moreover, since $Y_\eps\in \Gamma(x)$, we have $|Y_\eps - x|\approx \delta(Y_\eps)$. Consequently, by the triangle inequality, there exists a uniformly bounded constant $c \ge 1$ such that $x \in c Q_j$ (recall Notation \ref{notation:dyadic_cubes} (4)). Thus, we have $J_1^\eps(x) \le \sum_{j: \ell(Q_j) < \sqrt{\eps}} 1_{c Q_j}(x)$. By the Carleson packing condition of $\{Q_j\}_j$, we know that $\sum_{j} 1_{c Q_j}(x) < \infty$ for almost every $x$. Therefore $J_1^\eps(x)$ is bounded by the tail of a convergent series for almost every $x$, hence $J_1^\eps(x) \to 0$ as $\eps \to 0$ for almost every $x$.
               
  For $J_2^\eps$, we can use similar but simpler arguments as with $I_2^\eps$ in the previous case. Since $Y_\eps \in \Gamma(x) \cap t_{Q_j}$, there exists a subcube $Q \subset Q_j$ such that $\ell(Q) \approx \delta(Y_\eps) \lesssim \eps$ and $Y_\eps \in U_Q$. By definition, we have $Y_\eps \in \Gamma(y)$ and $\dist(y,Y_\eps) \lesssim \delta(Y_\eps)$ for every $y \in Q$. In particular, there exists a point $y \in Q_j$ such that $\dist(x,Q_j) \le \dist(x,y) \lesssim \eps \le \sqrt{\eps} \ell(Q_j)$. We now set
  \begin{align*}
    \widetilde{\Sigma}_j^\eps \coloneqq \{z \in \pom\setminus Q_j \colon \dist(z,Q_j) \lesssim \sqrt{\eps} \ell(Q_j)\}\,,
  \end{align*}
  and proceed as we did for $I_2^\eps$, but now using the exterior thin boundary estimate \eqref{dyadic_cubes:exterior_small_boundaries} in lieu of \eqref{dyadic_cubes:small_boundaries}.  We leave the remaining details to the reader.
\end{proof}

\begin{remark}
  The previous lemma is true also if we define the extension $F_0$ with respect to the overlapping boxes $T_Q$ or the tents $t_Q$ and in those cases the proof actually becomes simpler. However, in the next proof it is crucial that we use the modified Carleson tents.
\end{remark}

\begin{lemma}
  \label{lemma:dyadic_cme}
  The measure $|\nabla F_0(Y)| \, dY$ satisfies a quantitative codimension $1$ type Carleson measure estimate:
  \begin{align*}
    \sup_{r > 0, x \in \partial \Omega} \frac{1}{r^n} \iint_{B(x,r) \cap \Omega} |\nabla F_0(Y)| \, dY \lesssim C_0 \|f\|_{\BMO}
  \end{align*}
\end{lemma}

\begin{proof}
  It is easy to see that every ball $B(x,R) \cap \Omega$ with $x \in \pom$ and $R \lesssim \diam(\pom)$ can be covered by the union of interiors of a uniformly bounded number of Carleson boxes $T_Q$, with $R \approx \ell(Q)$ (see \cite[p. 2353--2354]{hofmannmartellmayboroda} for details). Thus, it is enough to show that
  \begin{align*}
    \iint_{\text{int}(T_Q)} |\nabla F_0(Y)| \, dY \lesssim C_0 \|f\|_\BMO\, \ell(Q)^n
  \end{align*}
  for an arbitrary cube $Q \in \D$. We consider first the case that $Q\subset Q_0$. Fix $Q\subset Q_0$ and a vector field $\oPsi \in C_0^1(\text{int}(T_Q))$ such that $\|\oPsi\|_{L^\infty} \le 1$. We have
  \begin{align*}
    \iint F_0 \, \text{div} \overrightarrow{\Psi}
        &= \sum_j \alpha_j \iint_{t_{Q_j}} \text{div} \overrightarrow{\Psi} \\
        &= \sum_{j \colon \ell(Q_j) < 2^M \ell(Q)} \alpha_j \iint_{t_{Q_j}} \text{div} \overrightarrow{\Psi}
           + \sum_{j\colon  \ell(Q_j) \geq 2^M \ell(Q)} \alpha_j \iint_{t_{Q_j}} \text{div} \overrightarrow{\Psi}
        \eqqcolon J_1 + J_2,
  \end{align*}
  where $M$ is a sufficiently large positive integer to be chosen. The sum $J_1$ is easy. Since $\partial t_{Q_j}\cap\Omega$ is a union of faces of Whitney cubes, and the support of $\oPsi$ has a strictly positive distance to $\pom$, we can apply the divergence theorem to get
  \begin{align}\label{eq9.7}
    \iint_{t_{Q_j}} \text{div} \overrightarrow{\Psi} 
    = \int_{\partial t_{Q_j} \cap \Omega} \overrightarrow{\Psi} \cdot \overrightarrow{N}
    \le \Hc^n(\partial t_{Q_j} \cap \Omega)
    \lesssim \ell(Q_j)^n\,,
  \end{align}
  where in the last step we have used Lemma \ref{lemma:adr_modified_tents}. Since $T_Q$ contains the support of $\oPsi$, every $Q_j$ appearing in $J_1$ is contained in a ball $B_Q^{**} \coloneqq B(x_Q,C\ell(Q))$, for some $C$ chosen large enough depending on $M$, $\eta$ and $K$. Combining the latter fact with \eqref{eq9.7}, and using the Carleson packing condition for the collection $\{Q_j\}$ (and Lemma \ref{lemma:carleson_sums} in the higher codimension case
  $d<n$), we see that $J_1 \lesssim C_0\, \ell(Q)^n$.
  
  The sum $J_2$ is little trickier. Since $\overrightarrow{\Psi}$ is compactly supported in 
  int($T_Q$), we have $\overrightarrow{\Psi} = 0$ on $\partial T_Q$. In particular, if we happen to have $T_Q = t_Q$, then $T_Q \cap t_{Q_j} = T_Q$ for every $Q_j$ in the sum $J_2$, and the same divergence theorem argument as above implies that $J_2 = 0$. Unfortunately, usually $t_Q \subsetneq T_Q$,  so we have to be more careful.
  
  By Lemma \ref{lemma:modified_tents_properties}, there is a collection  $\Fc(Q)=\{Q^i\}_{i=1}^N$, of uniformly bounded cardinality $N$, with $\ell(Q^{i'})=\ell(Q^i) \approx_{\eta,K} \ell(Q)$ for each $i,i'$, such that $\cup_i t_{Q^i}$ contains $T_Q$. We now choose $M=M(\eta,K)$ so that $\ell(Q^i) = 2^M\ell(Q)$, for every $Q^i\in \Fc(Q)$. Thus, the cubes $Q_j$ in $J_2$ satisfy $Q^i \cap Q_j \in \{\emptyset, Q^i\}$ for all $i$ and $j$. This choice and the divergence theorem give
  \begin{align*}
    J_2 = \sum_i \sum_{j: Q_j \supseteq Q^i} \alpha_j \iint_{t_{Q^i}} \text{div} \overrightarrow{\Psi}
        = \sum_i \sum_{j: Q_j \supseteq Q^i} \alpha_j \int_{\partial t_{Q^i} \cap \Omega} \overrightarrow{N} \cdot \overrightarrow{\Psi}\,,
  \end{align*}
  where we have used Lemma \ref{lemma:modified_tents_properties} ii) and iii). Since supp($\overrightarrow{\Psi}) \subset$ int($T_Q$), $\overrightarrow{\Psi}(X)$ can be non-zero only if $X$ lies in the interior of $T_Q$. Furthermore, the modified Carleson tents $t_{Q^i}$ are disjoint, and their union covers $T_Q$. Thus, for every point $X$ on $\partial t_{Q^i}$ where $\overrightarrow{\Psi}(X)$ is non-zero, there is a different cube $Q^k\in\Fc(Q)$ such that $X \in \partial t_{Q^k}$.
  
  By Lemma \ref{lemma:modified_tents_properties}, we have that $\partial t_{Q^i} \cap \partial t_{Q^k} \cap \Omega$ is either empty or it consists of a union of faces of Whitney cubes. Let us define the set of all the pairs of indices of the cubes $Q^i$ by setting
  \begin{align*}
    \Pc \coloneqq \{(i,k) \colon 1 \le i < k \le N\}
  \end{align*}
  and let us define the collection of the faces of Whitney cubes between $t_{Q^i}$ and $t_{Q^k}$ by setting
  \begin{align*}
    \Fs_{(i,k)} \coloneqq \{F \colon F \text{ is a face of a Whitney cube contained in } \partial t_{Q^i} \cap \partial t_{Q^k}\}
  \end{align*}
  for every $(i,k) \in \Pc$. Notice that $\Fs_{(i,k)}$ may be empty. We can now write
  \begin{align*}
   J_2 &= \sum_i \sum_{j: Q_j \supseteq Q^i} \alpha_j \int_{\partial t_{Q^i} \cap \Omega} \overrightarrow{N_i} 
    \cdot \overrightarrow{\Psi} \\[4pt]
       &= \sum_{(i,k) \in \Pc} \sum_{F \in \Fs_{(i,k)}} \left( \sum_{j : Q_j \supseteq Q^i} \alpha_j \int_{F} \overrightarrow{N_i} \cdot \overrightarrow{\Psi}
          + \sum_{j : Q_j \supseteq Q^k} \alpha_j \int_{F} \overrightarrow{N_k} \cdot \overrightarrow{\Psi} \right),
  \end{align*}
  where $\overrightarrow{N_i}$ is the outer unit normal of $\partial t_{Q^i} \cap \Omega$. We notice that on $F$ the normals $\overrightarrow{N_i}$ and $\overrightarrow{N_k}$ point to the opposite directions. Thus, we actually have
  \begin{align*}
    J_2 = \sum_{(i,k) \in \Pc} \sum_{F \in \Fs_{(i,k)}} \left( \sum_{j : Q_j \supseteq Q^i} \alpha_j - \sum_{j : Q_j \supseteq Q^k} \alpha_j \right) \int_{F} \overrightarrow{N_i} \cdot \overrightarrow{\Psi}.
  \end{align*}
  By Lemma \ref{lemma:bounded_sum_difference}, 
  we therefore have
  \begin{equation*}
    |J_2| \lesssim C_0 \|f\|_{\BMO} \sum_{(i,k) \in \Pc} \sum_{F \in \Fs_{(i,k)}} \Hc^n(F).
  \end{equation*}
  Furthermore, if $F \in \Fs_{(i,k)}$, then by definition $F\subset  \partial t_{Q^i} \cap \partial t_{Q^k}$, so 
  \begin{equation*}
    \sum_{F \in \Fs_{(i,k)}} \Hc^n(F) \le \Hc^n(\partial t_{Q^i}  \cap \Omega)
    \lesssim \ell(Q)^n,
  \end{equation*}
  since $\partial t_{Q^i} \cap \Omega$ is upper $n$-ADR by Lemma \ref{lemma:adr_modified_tents}, and $\ell(Q) \approx \ell(Q^i)$. The 
  number of the modified Carleson tents $t_{Q^i}$ was uniformly bounded, hence, so is the cardinality of the set $\Pc$. 
  Thus, $J_2 \lesssim C_0 \|f\|_{\BMO}\, \ell(Q)^n$. 
  This completes the proof in the case $Q \subset Q_0$.
  
  Next, we suppose that $Q \not\subset Q_0$.  As above, let $\oPsi \in C_0^1(\text{int}(T_Q))$ be a vector field  such that $\|\oPsi\|_{L^\infty} \le 1$. Consider first the case that
  $\ell(Q) \geq \ell(Q_0)$. We then have
  \begin{align*}
   \left| \iint F_0 \, \text{div} \overrightarrow{\Psi}\right|
        = \left|\sum_j \alpha_j \iint_{t_{Q_j}} \text{div} \overrightarrow{\Psi}\right| 
        &= \left|\sum_{j} \alpha_j  \int_{\partial t_{Q_j} \cap \Omega} 
        \overrightarrow{\Psi} \cdot \overrightarrow{N}\right| \\[4pt]
        &\le \sup_j|\alpha_j|\sum_j\Hc^n(\partial t_{Q_j} \cap \Omega) \\[4pt]
        &\lesssim  \|f\|_{\BMO} \sum_j \ell(Q_j)^n \\[4pt]
        &\lesssim C_0  \|f\|_{\BMO}\, \ell(Q_0)^n \\[4pt]
        &\lesssim  C_0 \|f\|_{\BMO}\,\ell(Q)^n\,,
  \end{align*}
  where we have used the upper-ADR property of $\partial t_{Q_j}$, the uniform bound for $|\alpha_j|$, the packing condition for the collection $\widetilde{\D}_{Q_0} =\{Q_j\}_j$, the fact that $Q_0$ contains every $Q_j$ (and Lemma \ref{lemma:carleson_sums} in the higher codimension case $d<n$).

  We therefore suppose that $\ell(Q)<\ell(Q_0)$.  In this case, since $\text{supp}(\oPsi) \subset T_Q$, and since $t_{Q_j}\subset t_{Q_0} = T_{Q_0}$ for each $Q_j$, we may further assume that $\dist(Q,Q_0) \lesssim \ell(Q)$, where the implicit constants depend on $\eta$ and $K$, otherwise 
  $\iint F_0 \, \text{div} \overrightarrow{\Psi}$ vanishes.  In particular, we may suppose that $Q\subset P\in\D$, where $\ell(P) =\ell(Q_0)$, and $\dist(P,Q_0) \lesssim \ell(Q_0)$. Let us enumerate the collection of such $P$ (with $Q_0$ itself excluded), as $\{P_m\}_{m=1}^N$, where $N$ is a uniformly bounded number depending only upon $n$, ADR, $\eta$ and $K$. For each such $P_m$, we construct pairwise disjoint $\{t_{Q'}\}_{Q'\subset P_m}$ exactly as we constructed $t_{Q'}$ for $Q' \subset Q_0$ in Section \ref{section:carleson_regions}, and then we build disjoint $\{t^*_{Q'}\}_{Q'\in \D_{Q_0} \cup \D_{P_1}\cup...\cup\D_{P_N}}$ by setting
  \begin{alignat*}{3}
    t^*_{Q'} &\coloneqq t_{Q'}, &&Q'\subset Q_0 \\
    t^*_{Q'} &\coloneqq t_{Q'} \setminus t_{Q_0},  &&Q'\subset P_1\\
    t^*_{Q'} &\coloneqq t_{Q'} \setminus (t_{Q_0} \cup t_{P_1}), &&Q'\subset P_2\\
    &  \qquad \vdots &&\qquad \vdots \\
    t^*_{Q'} &\coloneqq t_{Q'} \setminus (t_{Q_0} \cup t_{P_1}\cup \ldots \cup t_{P_{N-1}}), \qquad &&Q'\subset P_N.
  \end{alignat*}
  We may then generalize Lemma \ref{lemma:modified_tents_properties}, so that in particular, for each $Q\in  \D_{P_1}\cup...\cup\D_{P_N}$, there is a collection $\Fc(Q)=\{Q^i\}_i \subset \D_{Q_0} \cup \D_{P_1}\cup...\cup\D_{P_N}$, of uniformly bounded cardinality depending only on $n$, ADR, $\eta$ and $K$, such that $\ell(Q^i) \approx \ell(Q)$, with
  $\ell(Q^i) = \ell(Q^{i'})$ for all $i, i'$, and $T_Q \subset \bigcup_i t^*_{Q^i}$. Moreover,  $t^*_{Q'} \subset t^*_{Q''}$, provided that $Q'\subset Q''$, and $t^*_{Q'} \cap t^*_{Q''}=\emptyset$ whenever  $Q'\cap Q''=\emptyset$. One may now repeat the previous argument, {\em mutatis mutandis}, noting that Lemma \ref{lemma:bounded_sum_difference} still applies in the case that $Q^i\cap Q_0 = \emptyset$.  We omit the details.
\end{proof}

Let $F$ be the regularization of $F_0$, as in Section \ref{section:eps-approximability}. 
By Lemmas \ref{lemma:smooth_regularization} and 
\ref{lemma:gradient_regularization}, 
 $F\in C^\infty(\Omega)$, and satisfies the pointwise 
 gradient bound.  We also find that $F$ converges non-tangentially to $f$ almost everywhere,  and that $|\nabla F(Y)| \, dY$ is a Carleson measure,
 by combining Lemmas \ref{lemma:dyadic_extension_nt} and \ref{lemma:smooth_nt_convergence}, 
and Lemmas \ref{lemma:dyadic_cme} and \ref{lemma:carleson_regularization}.
This completes the proof of Proposition \ref{proposition:dyadic_extension}.

\section{Garnett's decomposition lemma and proof of Theorem \ref{theorem:varopoulos_extension}}
\label{section:garnett}

In this last section, we present the final ingredient for the proof of Theorem \ref{theorem:varopoulos_extension}: a straightforward generalization of Garnett's decomposition lemma to the setting of ADR sets. The proof follows the original argument sketched as an exercise in Garnett \cite[Section VI, Exercise 12 (c)]{garnett} (and stated without proof in \cite[Lemma 1.2.1]{varopoulos1}). We include the details here for the sake of completeness.  

\begin{lemma}[Garnett's lemma]
  \label{lemma:garnett}
  Let $E\subset \ree$ be a $d$-ADR set, $d\leq n$.  Let $Q_0 \in \D$, and consider $f\in  \BMO_\D(E$) (see
  Definition \ref{bmodef}), which vanishes on $E\setminus Q_0$ (provided the latter is non-empty). Then there 
  is a collection $\widetilde{\D}_{Q_0} = \{Q_j\}_j \subset \D_{Q_0}$ and coefficients $\alpha_j$ such that
  \begin{enumerate}
    \item[(1)] $\sup_j |\alpha_j| \lesssim \|f\|_{\BMO_\D}$,
    \item[(2)] $f - \langle f \rangle_{Q_0} 
    = \widetilde{f} + \sum_j \alpha_j 1_{Q_j}$, 
    where $\widetilde{f}\in L^\infty(E,d\sigma)$ with $\|\widetilde{f}\|_{L^\infty} \lesssim \|f\|_{\BMO_\D}$,
    \item[(3)] $\widetilde{\D}_{Q_0}$ satisfies a Carleson packing condition with $\Cs_{\widetilde{\D}_{Q_0}} \lesssim 1$.
  \end{enumerate}
\end{lemma}

\begin{remark}
Since $\|f\|_{\BMO_\D}\lesssim  \|f\|_{\BMO}$, the Lemma holds of course for $f\in \BMO(E)$.
\end{remark}

\begin{remark}\label{r10.3}
  If $Q_0\subsetneq E$, then there is a cube $Q_1$ disjoint from $Q_0$, of the same dyadic generation (i.e., such that $\ell(Q_1) = \ell(Q_0)$), with common dyadic ancestor $Q_*$, such that $\dist(Q_0,Q_1) \lesssim \ell(Q_0) = \ell(Q_1)\approx \ell(Q_*)$. Since $f$ vanishes outside of $Q_0$, we have that $f\equiv 0$ on $Q_1$, hence
  \begin{align*}
    |\langle f \rangle_{Q_0}| =|  \langle f \rangle_{Q_0}-  \langle f \rangle_{Q_1}| \lesssim \|f\|_{\BMO_\D}\,,
  \end{align*}
  where the last inequality is a well-known fact about dyadic BMO.  Consequently, in this case we may absorb $\langle f \rangle_{Q_0}$ into $\widetilde{f}$, so that item (2) in Lemma \ref{lemma:garnett} becomes
  \begin{enumerate}
    \item[(2a)] $f = \widetilde{f} + \sum_j \alpha_j 1_{Q_j}$, where $\widetilde{f}\in L^\infty(E,d\sigma)$, with $\|\widetilde{f}\|_{L^\infty} \lesssim \|f\|_{\BMO_\D}$.
  \end{enumerate}
  Observe also that in this case  
    $\widetilde{f}$ vanishes on $E\setminus Q_0$. 
\end{remark}

\begin{proof}[{Proof of Lemma \ref{lemma:garnett}}] 
  We build the collection $\widetilde{\D}_{Q_0}$ by using a stopping time argument. Set $\Fc_0 = \{Q_0\}$. We have
  \begin{align*}
    \langle |f - \langle f \rangle_{Q_0}| \rangle_{Q_0} \le \|f\|_{\BMO_\D}\,.
  \end{align*}
  Let us subdivide $Q_0$ and stop when $\langle |f - \langle f \rangle_{Q_0}| \rangle_Q > 2 \|f\|_{\BMO_\D}$. We let $\Fc_1 = \{Q_j^{(1)}\}_j$ be the collection of the maximal stopping cubes. By definition, 
  \begin{align*}
    \langle |f- \langle f \rangle_{Q_j^{(1)}}| \rangle_{Q_j^{(1)}} \le \|f\|_{\BMO_\D}\,,\quad \forall \, Q_j^{(1)} \in \Fc_1\,.
  \end{align*}
  
  For each $Q_j^{(1)}$, we repeat the process with the modified stopping condition 
  \begin{align*}
    \langle |f- \langle f \rangle_{Q_j^{(1)}}| \rangle_Q > 2 \|f\|_{\BMO_\D}\,.
  \end{align*}
  We let $\Fc_2 = \{Q_j^{(2)}\}_j$ be the collection of maximal stopping cubes. Again by definition, 
  \begin{align*}
    \langle |f- \langle f \rangle_{Q_j^{(2)}}| \rangle_{Q_j^{(2)}} \le \|f\|_{\BMO_\D}\,.
  \end{align*}
  
  We continue in this way, and denote the collection of cubes of level $i$ by $\Fc_i$. We now set $\widetilde{\D}_{Q_0} \coloneqq \bigcup_i \Fc_i$, and define
  \begin{align*}
    \alpha_j^{(i)} \coloneqq \langle f - \langle f \rangle_{P_{k(j)}^{(i-1)}} \rangle_{Q_j^{(i)}} = \langle f \rangle_{Q_j^{(i)}} - \langle f \rangle_{P_{k(j)}^{(i-1)}},
  \end{align*}
  where for $i\geq 1$, $P_{k(j)}^{(i-1)}$ is the unique cube in $\Fc_{i-1}$ such that $Q_j^{(i)} \subset P_{k(j)}^{(i-1)}$. We prove the properties (1) -- (3) in order.
  
  \smallskip

  \noindent (1) Property (1) follows easily from the ADR property and the stopping criterion:
  \begin{align*}
    |\alpha_j^{(i)}| = \left| \fint_{Q_j^{(i)}} f - \langle f \rangle_{P_{k(j)}^{(i-1)}} \, d\sigma \right|
                     &\le \fint_{Q_j^{(i)}} \left| f - \langle f \rangle_{P_{k(j)}^{(i-1)}} \right| \, d\sigma \\
                     &\lesssim \fint_{\widetilde{Q}_j^{(i)}} \left| f - \langle f \rangle_{P_{k(j)}^{(i-1)}} \right| \, d\sigma
                      \lesssim \|f\|_{\BMO_\D},
  \end{align*}
  where $\widetilde{Q}_j^{(i)}$ is the dyadic parent of $Q_j^{(i)}$.
  
  \smallskip
  
  \noindent (2) Observe that $f-\langle f \rangle_{Q_0} = -\langle f \rangle_{Q_0}$ in $E\setminus Q_0$, if the latter is non-empty, and  in this case, by Remark \ref{r10.3} we may simply set $\widetilde{f} = -\langle f \rangle_{Q_0}$ on $E\setminus Q_0$. It is therefore enough to prove the decomposition (2) on $Q_0$.

  For $x\in Q_0$, we define a counting function
  \begin{align*}
    \Nc_f(x) \coloneqq \#\left\{i\geq 1: \exists Q^{(i)}_j \in \Fc_i \text{ with } x\in Q^{(i)}_j\right\}\,.
  \end{align*}
  If $\Nc_f(x)<\infty$,  we set $N_x \coloneqq \Nc_f(x)$, and note that in this case there is a cube $Q_{\text{min}}(x) \in \Fc_{N_x}$ such that $x \in Q_{\text{min}}(x)$, and $x \notin Q_j^{(i)}$ for all $i > N_x$ and every $j$. Also, for every $i \le N_x$, there now exists a cube $Q_{j(i,x)}^{(i)} \in \Fc_i$ such that $x \in Q_{j(i,x)}^{(i)}$. Since the cubes in each $\Fc_i$ are disjoint, by the definition of the cubes $P_{k(j)}^{(i-1)}$, we have
  \begin{align*}
    \alpha_{j(i,x)}^{(i)} = \langle f \rangle_{Q_{j(i,x)}^{(i)}} - \langle f \rangle_{P_{k(j(i,x))}^{(i-1)}}
                      = \langle f \rangle_{Q_{j(i,x)}^{(i)}} - \langle f \rangle_{Q_{j(i-1,x)}^{(i-1)}}.
  \end{align*}
  In particular, the sum $\sum_{i=1}^{N_x} \alpha_{j(i,x)}^{(i)}$ is telescoping and we get
  \begin{align*}
    \sum_{i,j} \alpha_j^{(i)} 1_{Q_j^{(i)}}(x) 
      = \sum_{i=1}^{N_x} \alpha_{j(i,x)}^{(i)}
      = -\langle f \rangle_{Q_0} + \langle f \rangle_{Q_{\text{min}}(x)} \,.
  \end{align*}
  
  On the other hand, if  $\Nc_f(x)= \infty$, then the analogous telescoping sum becomes
  \begin{align*}
    \sum_{i,j} \alpha_j^{(i)} 1_{Q_j^{(i)}}(x)
      = \sum_{i=1}^{\infty} \alpha_{j(i,x)}^{(i)}
      = -\langle f \rangle_{Q_0} + f(x) \,,
  \end{align*}
  by Lebesgue's differentiation theorem, where the latter identity is valid for $\sigma$-a.e.\ $x$ such that $\Nc_f(x)$ is infinite. Setting
  \begin{equation*}
    \widetilde{f}(x) \coloneqq \left\{
      \begin{array}{cl}
        f(x) - \langle f \rangle_{Q_{\text{min}}(x)}\,, \,\,&\text{if } \Nc_f(x)<\infty \\[6pt]
        0\,, \,\,&\text{if } \Nc_f(x)=\infty\,,
      \end{array} \right.\,.
  \end{equation*}
  we obtain the claimed decomposition in (2). It remains to check that with this definition, we have $\|\widetilde{f}\|_{L^\infty(E,d\sigma)}\lesssim \|f\|_{\BMO_\D}$.  To this end, observe that in order to have $\Nc_f(x)<\infty$,  we must have that for every dyadic cube $Q$ with $x\in Q\subsetneq Q_{\text{min}}(x)$,
  \begin{align*}
    \left\langle \left| f - \langle f \rangle_{Q_{\text{min}}(x)}\right|\right\rangle_Q \leq 2\|f\|_{\BMO_\D}\,,
  \end{align*}
  otherwise, there would have been another stopping cube containing $x$, and strictly contained in $Q_{\text{min}}(x)$, which contradicts the definition of $Q_{\text{min}}(x)$.  By Lebesgue's differentiation theorem, we therefore find that $|\widetilde{f}(x)|\leq 2\|f\|_{\BMO_\D}$ for $\sigma$-a.e.\ $x$ such that $\Nc_f(x)<\infty$, so that (2) holds.
  
  \smallskip

  \noindent (3) By a standard limiting argument, we may assume that the collection $\widetilde{\D}_{Q_0}$ is finite. 
  We first notice that by the stopping conditions we have
  \begin{align}
    \label{estimate:measure_of_stopping_cube}
    \sigma(Q_j^{(i)}) \le \frac{1}{2 \|f\|_{\BMO_\D}} \int_{Q_j^{(i)}} |f - \langle f \rangle_{P_{k(j)}^{(i-1)}}| \, d\sigma.
  \end{align}
  Let $Q \subseteq Q_0$ be fixed. We set
  \begin{align*}
    I \coloneqq \sum_{R \in \widetilde{\D}_{Q_0}, R \subseteq Q} \sigma(R) = I_1 + I_2,
  \end{align*}
  where $I_1$ is the sum over those $Q_j^{(i)}$ such that $P_{k(j)}^{(i-1)} \subset Q$ and 
  $I_2$ is the sum over the rest of the relevant cubes. The cubes in the sum $I_2$ are disjoint and thus, $I_2 \le \sigma(Q)$. Let $i(Q)$ be the smallest integer such that $\Fc_{i(Q)}$ contains at least one cube in the sum $I_1$; thus, $I_2$ is the sum over the cubes in $\Fc_{i(Q)-1}$ that are contained in $Q$. With this notation, we may write 
  \begin{align*}
    I = \sum_{i\geq i(Q)-1}\sum_{R \in \Fc_i, R \subset Q} \sigma(R) 
      = \sum_{i\geq i(Q)} \sum_{R \in \Fc_i, R \subset Q} \sigma(R) + I_2
      = I_1 + I_2.
  \end{align*}
  We have
  \begin{align*}
    I_1 &= \sum_{i\geq i(Q)}\ \sum_{j:\,Q_j^{(i)} \in \Fc_i, Q_j^{(i)} \subset Q} \sigma(Q_j^{(i)}) \\
        &\overset{\eqref{estimate:measure_of_stopping_cube}}{\le} \frac{1}{2\|f\|_{\BMO_\D}} \sum_{i\geq i(Q)} \ 
        \sum_{j:\,Q_j^{(i)} \in \Fc_i, Q_j^{(i)} \subset Q} 
        \int_{Q_j^{(i)}} |f - \langle f \rangle_{P_{k(j)}^{(i-1)}}| \, d\sigma \\
        &\overset{\text{(A)}}{=} \frac{1}{2\|f\|_{\BMO_\D}} 
        \sum_{i\geq i(Q)} \ 
        \sum_{j:\,Q_j^{(i-1)} \in \Fc_{i-1}, Q_j^{(i-1)} \subset Q} \ 
        \sum_{l:\,Q_{l}^{(i)} \in \Fc_i, Q_{l}^{(i)} \subset Q_j^{(i-1)}} 
        \int_{Q_{l}^{(i)}} |f - \langle f \rangle_{Q_{j}^{(i-1)}}| \, d\sigma \\
        &\overset{\text{(B)}}{\le} \frac{1}{2\|f\|_{\BMO_\D}} \ 
        \sum_{i\geq i(Q)} \ 
        \sum_{j:\,Q_j^{(i-1)} \in \Fc_{i-1}, Q_j^{(i-1)} \subset Q} 
        \int_{Q_j^{(i-1)}} |f - \langle f \rangle_{Q_{j}^{(i-1)}}| \, d\sigma \\                                
        &= \frac{1}{2\|f\|_{\BMO_\D}} 
        \sum_{i\geq i(Q)} \ 
        \sum_{j:\,Q_j^{(i-1)} \in \Fc_{i-1}, Q_j^{(i-1)} \subset Q} \sigma(Q_j^{(i-1)}) \left\langle |f - \langle f \rangle_{Q_{j}^{(i-1)}}| \right\rangle_{Q_j^{(i-1)}} \\
        &\overset{\text{(C)}}{\le} \frac{1}{2} \sum_{i\geq i(Q)} \ 
        \sum_{j:\,Q_j^{(i-1)} \in \Fc_{i-1}, Q_j^{(i-1)} \subset Q} \sigma(Q_j^{(i-1)}) \\
        &\le \frac{1}{2} I,
  \end{align*}
  where we used in (A) the observation that with this notation $P_{k(l)}^{(i-1)} = Q_j^{(i-1)}$, in (B) the fact that the cubes $Q_{l(j)}^{(i)} \in \Fc_i$ are disjoint, and in (C) the definition of the $\BMO_\D$ norm. In particular,
  \begin{align*}
    I = I_1 + I_2 \le \frac{1}{2} I + \sigma(Q)
  \end{align*}
  and thus $I \le 2 \sigma(Q)$. This completes the proof.
\end{proof}

Theorem \ref{theorem:varopoulos_extension} follows now easily from the other results we have proven:

\begin{proof}[{Proof of Theorem \ref{theorem:varopoulos_extension}}]
  Suppose that $f$ is a compactly supported function in $\BMO(\pom)$. Then, by Theorem \ref{theorem:existence_of_dyadic_cubes}, there is a choice of dyadic system
  $\D$ such that there exists a cube $Q_0 \in \D$ with $\text{supp} \, f \subset Q_0$. 
  By Lemma \ref{lemma:garnett}, there exists now a decomposition $f = \widetilde{f} + f_0$, where
  \begin{enumerate}
    \item[(1)] $\widetilde{f}$ is bounded $\sigma$-a.e., 
    and $\| \widetilde{f} \|_{L^\infty(\pom)} \lesssim \|f\|_{\BMO(\pom)}$, and
    \item[(2)] $f_0(x) = \sum_{Q \in \widetilde{\D}_{Q_0}} \alpha_Q 1_Q(x)$ for a collection $\widetilde{\D}_{Q_0} \subset \D_{Q_0}$ and coefficients $\alpha_Q$ such that
    \begin{enumerate}
      \item[$\bullet$] $\Cs_{\widetilde{\D}_{Q_0}} \lesssim 1$, and
      \item[$\bullet$] $\sup_{Q \in \widetilde{\D}_{Q_0}} |\alpha_Q| \lesssim \|f\|_{\BMO(\pom)}$.
    \end{enumerate}
  \end{enumerate}
  By Theorem \ref{theorem:bounded_extension}, we know that there exists a function $\Phi \in C^\infty(\Omega)$ such that $\Phi$ converges to $\widetilde{f}$ non-tangentially almost everywhere, the measure $\mu_1 \coloneqq |\nabla \Phi(Y)| \, dY$ is a Carleson measure and $C_{\mu_1} \lesssim \|\widetilde{f}\|_{L^\infty(\pom)} \lesssim \|f\|_{\BMO(\pom)}$.
  
  By the decomposition $f = \widetilde{f} + f_0$, we know that $f_0$ is a BMO function as it is a sum of two BMO functions. Thus, by Proposition \ref{proposition:dyadic_extension}, there exists a function $F \in C^\infty(\Omega)$ such that $F$ converges to $f_0$ non-tangentially almost everywhere, the measure $\mu_2 \coloneqq |\nabla F(Y)| \, dY$ is a Carleson measure and
  \begin{align*}
    C_{\mu_2} \lesssim \Cs_{\widetilde{\D}_{Q_0}} \|f_0\|_{\BMO(\pom)} 
              &\lesssim \|\widetilde{f}\|_{\BMO(\pom)} + \|f\|_{\BMO(\pom)} \\
              &\lesssim \|\widetilde{f}\|_{L^\infty(\pom)} + \|f\|_{\BMO(\pom)} 
              \lesssim \|f\|_{\BMO(\pom)}.
  \end{align*}
  Thus, we can set $V \coloneqq \Phi + F$.
\end{proof}

\bibliography{varopoulos_extension}
\bibliographystyle{alpha}

\end{document}